\RequirePackage{fix-cm}
\documentclass[smallextended]{svjour3}       
\smartqed  
\usepackage{graphicx}

\usepackage{hyperref}
\usepackage{amssymb}
\usepackage{graphicx}
\usepackage{amsfonts}
\usepackage{mathrsfs}
\usepackage{amsmath}
\usepackage{amssymb}
\usepackage{tikz} 
\usepackage{csquotes}
\usepackage{subcaption}
\usepackage{algorithm}
\usepackage[noend]{algpseudocode}
\usepackage{bbm}
\usepackage{booktabs}
\usepackage[title]{appendix}
\usetikzlibrary{decorations.pathreplacing,calligraphy}
\newcommand*\circled[1]{\tikz[baseline=(char.base)]{
            \node[shape=circle,draw,inner sep=2pt] (char) {#1};}}

\newcommand{\R}{\mathbb{R}}
\newcommand{\N}{\mathbb{N}}
\newcommand{\Z}{\mathbb{Z}}

\newcommand{\conv}{\operatorname{conv}}

\newcommand{\F}{\mathscr{F}}

\newcommand{\NP}{$\mathcal{NP}$}
\newcommand{\transp}{\mathsf T}

\newenvironment{cpf}
{\begin{trivlist} \item[] {\em Proof of claim }}
{$\hfill\diamond$ \end{trivlist}}

\begin{document}

\title{On the Complexity of Separating Cutting Planes for the Knapsack Polytope}

\author{Alberto Del Pia \and Jeff Linderoth \and Haoran Zhu\footnote{Corresponding author: Haoran Zhu, hzhu94@wisc.edu}
}

\authorrunning{A. Del Pia et al.} 

\institute{Haoran Zhu \at
              Department of Industrial and Systems Engineering, University of Wisconsin-Madison, \\
              Madison, WI, 53705, USA\\
              hzhu94@wisc.edu           
}

\institute{Department of Industrial and Systems Engineering \& Wisconsin Institute for Discovery, University of Wisconsin-Madison, Madison, USA, \email{delpia@wisc.edu}
\email{linderoth@wisc.edu}
\email{hzhu94@wisc.edu}}

\date{Received: date / Accepted: date}

\maketitle

\begin{abstract} 
We close three open problems on the separation complexity of valid inequalities for the knapsack polytope. 
Specifically, we establish that the separation problems for extended cover inequalities, $(1,k)$-configuration inequalities, and weight inequalities are all \NP-complete.  
We also show that, when the number of constraints of the LP relaxation is constant and its optimal solution is an extreme point, then the separation problems of both extended cover inequalities and weight inequalities can be solved in polynomial time. 
Moreover, we provide a natural generalization of $(1,k)$-configuration inequality which is easier to separate and contains the original $(1,k)$-configuration inequality as a strict sub-family.


\end{abstract}

\section{Introduction}

The \emph{multi-dimensional knapsack problem} is the integer programming (IP) problem
\begin{equation}
\label{eq:knapsack}
\max\{c^\transp x :  Ax \leq d, \ x \in \{0,1\}^n\},
\end{equation}
where $A \in \Z^{m \times n}_+$, $c \in \Z^n_+$, and $d \in \Z^m_+$.
When the constraint matrix $A$ only has one row $a$ and the right-hand side vector is a positive integer $b$, problem~\eqref{eq:knapsack} is referred to as \emph{knapsack problem}, and the convex hull of the associated feasible region, $\conv(\{x \in \{0,1\}^n : a^\transp x \leq b\})$, is referred to as the \emph{knapsack polytope}.

%
The multi-dimensional knapsack problem is a fundamental problem in discrete optimization, and valid inequalities for the feasible region have been widely studied, see, e.g., \cite{kellerer2004multidimensional,puchinger2010multidimensional,10.1007/978-3-030-73879-2_14} and the modern survey \cite{hojny2019knapsack}.
In this paper, we study the complexity of the separation problem for well-known families of valid inequalities for \eqref{eq:knapsack}.

A standard and computationally useful way for generating cuts for \eqref{eq:knapsack} is to construct cuts for the knapsack polytope defined by its individual constraints. 
Suppose $a$ is a row of the constraint matrix $A$, and let $b$ be the corresponding coordinate of the right-hand side $d$.  We denote the associated knapsack polytope by $K: = \conv(\{x \in \{0,1\}^n : a^\transp x \leq b\})$. 

Many families of valid inequalities for $K$ are based on the notion of a \emph{cover}, which is a subset $C$ of $\{1,2,\ldots, n\}$ such that $\sum_{i \in C} a_i > b$.
Given a cover $C$, the inequality
\begin{equation*}
\sum_{i \in C} x_i \leq |C| - 1
\end{equation*}
is valid for $K$, and is called a \emph{cover inequality} (CI).  Cover inequalities can often be strengthened through a process called lifting, and the resulting inequalities are called \emph{lifted cover inequalities} (LCIs) \cite{balas1978facets,MR1656936,letchford2019lifted,wolsey1975faces,padberg1975note}.
Balas \cite{balas1975facets} gave one family of LCIs  known as \emph{extended cover inequality} (ECI), which have the form
\begin{equation*}
\sum_{j \notin C: a_j \geq \max_{i \in C} a_i} x_j +  \sum_{i \in C} x_i \leq |C| - 1.
\end{equation*}
A \emph{minimal cover} is a cover $C$ such that $\sum_{i \in C \setminus \{j\}} a_i \leq b$ for any $j \in C$.
A set $N \cup \{t\}$ with $N \subsetneq \{1, \ldots, n\}$ and $t \notin N$ is called a \emph{$(1,k)$-configuration} for $k \in \{2, \ldots, |N|\}$ if $\sum_{i \in N} a_i \leq b$ and $Q \cup \{t\}$ is a minimal cover for every $Q \subseteq N$ with $|Q| = k$. Padberg \cite{padberg19801} showed that for any $(1,k)$-configuration $N \cup \{t\}$, the inequality
\begin{equation*}
(|S| - k + 1) x_t + \sum_{i \in S} x_i \leq |S|
\end{equation*}
is valid for $K$ for every $S \subseteq N$ with $|S| \geq k$.  This inequality is called a \emph{$(1,k)$-configuration inequality}. 

Other valid inequalities for the knapsack polytope $K$ arise from the concept of a pack. 
For the knapsack polytope $K$, a set $P \subseteq \{1, \ldots, n\}$ is a \emph{pack} if $\sum_{i \in P} a_i \leq b$.
Given a pack $P$, the corresponding \emph{pack inequality} $\sum_{i \in P}a_i x_i \leq \sum_{i \in P} a_i$ is trivially valid for $K$, as it is implied by the upper bound constraints $x_i \leq 1$. 
However, pack inequalities can be lifted in several different ways to obtain more interesting \emph{lifted pack inequalities (LPIs)} \cite{atamturk2005cover}. 
Weismantel \cite{weismantel19970} derived the weight-inequalities, which are LPIs.  To define the weight inequalities, let $r(P): = b - \sum_{i \in P} a_i$ be the \emph{residual capacity} of the pack $P$. 
The indices $j \notin P$ with $a_j > r(P)$ are lifted to obtain the \emph{weight inequality (WI)}:
\begin{equation*}
\sum_{i \in P} a_i x_i + \sum_{j \notin P} \max\{a_j-r(P),0\} x_j \leq \sum_{i \in P} a_i.
\end{equation*}


Consider the \emph{linear programming (LP) relaxation} of \eqref{eq:knapsack}:
\begin{equation}
\label{eq:knapsack_relaxation}
\max\{c^\transp x :  Ax \leq d, \ x \in [0,1]^n\}.
\end{equation}
For a given family $\F$ of valid inequalities for \eqref{eq:knapsack}, the associated \emph{separation problem} is defined as follows: ``Let $x^*$ be a feasible solution to \eqref{eq:knapsack_relaxation}, does there exist an inequality in $\F$ that is violated by $x^*$? 
If so, return one such inequality from $\F$.''
In this paper, we are mainly interested in the weaker decision version of the separation problem where we do not have to return a separating inequality, even if it does exist, and we assume that $x^*$ is an optimal solution to \eqref{eq:knapsack_relaxation}.
In fact, the separation of an optimal solution is not harder than the separation of a general feasible solution, and from a computational point of view, $x^*$ almost always corresponds to an optimal solution to some linear relaxation.

The separation problem for several families of valid inequalities for the knapsack polytope has been shown to be \NP-complete, including 
CIs \cite{klabjan1998complexity},
and LCIs \cite{MR1688131}.
On the other hand, the complexity of the separation problem for extended cover inequalities, $(1,k)$-configuration inequalities, and weight inequalities are, to the best of our knowledge, unknown.
Kaparis and Letchford stated that the separation problem seems likely to be \NP-hard for ECIs in \cite{kaparis2010separation}.
It was conjectured explicitly in \cite{ferreira1996solving} that the separation problem for $(1,k)$-configuration inequalities is \NP-hard. 
Moreover, the complexity of the separation problem for WIs is also open, as mentioned in \cite{hojny2019knapsack}. 
In this paper, we provide positive answers to all these conjectures. 
Namely, we show that the separation problems for ECIs, for $(1,k)$-configuration inequalities, and for WIs are all \NP-complete. 
The first two results are proven via a reduction from the separation problem for CIs, and the separation complexity for WIs is given via a reduction from the Subset Sum Problem (SSP).
 
Along with these \NP-hardness results, we also present some positive results for the separation problems of those cutting-planes. 
Specifically, we show that when the number of constraints of the LP relaxation~\eqref{eq:knapsack_relaxation} is constant, and the optimal solution $x^*$ is an extreme point, then the separation problems for ECIs, $(1,k)$-configuration inequalities, and WIs, are all polynomial-time solvable. See Corollary~\ref{cor: 1}, Corollary~\ref{cor: config}, and Corollary~\ref{cor: 2}. 
 
We remark that several heuristics and exact separation algorithms are present in the literature for these families of cuts.  Both Gabrel and Minoux~\cite{gabrel2002scheme} and Kaparis and Letchford~\cite{kaparis2010separation} provide an exact separation algorithm for ECIs that runs in pseudo-polynomial time. Ferreira et al.~\cite{ferreira1996solving} presented simple heuristics for the separation problem of $(1,k)$-configuration inequalities.
For the separation problem for WIs, Weismantel~\cite{weismantel19970} proposed an exact algorithm that runs in pseudo-polynomial time. 
Helmberg and Weismantel~\cite{MR1607353} presented a fast separation heuristic for WIs that simply inserts items into the pack $P$ in non-increasing order of $x^*$ value.  
Kaparis and Letchford~\cite{kaparis2010separation} gave two exact algorithms and a heuristic for separating WIs and show how to convert these methods into heuristics for separating LPIs.

Next, we formally define the separation problems considered in this paper.
\smallskip

\noindent \textbf{Problem CI-SP}\\
\textbf{Input: }$(A, d, c) \in (\Z^{m \times n}_+, \Z^m_+, \Z^n_+)$ and an optimal solution $x^*$ to the LP relaxation \eqref{eq:knapsack_relaxation}.\\
\textbf{Question: }Is there a cover $C$ with respect to some row constraint $a^\transp x \leq b$ of \eqref{eq:knapsack_relaxation}, such that $\sum_{i \in C}x^*_i > |C| - 1$?\\

\noindent \textbf{Problem ECI-SP}\\
\textbf{Input: } $(A, d, c) \in (\Z^{m \times n}_+, \Z^m_+, \Z^n_+)$ and an optimal solution $x^*$ to the LP relaxation \eqref{eq:knapsack_relaxation}.\\
\textbf{Question: } Is there a cover $C$ with respect to some row constraint $a^\transp x \leq b$ of \eqref{eq:knapsack_relaxation}, such that $\sum_{j \notin C: a_j \geq \max_{i \in C} a_i} x_j +  \sum_{i \in C} x_i > |C| - 1$? \\


\noindent\textbf{Problem CONFIG-SP} \\
\textbf{Input: } $(A, d, c) \in (\Z^{m \times n}_+, \Z^m_+, \Z^n_+)$ and an optimal solution $x^*$ to the LP relaxation \eqref{eq:knapsack_relaxation}.\\
\textbf{Question: } Is there a $(1,k)$-configuration $N \cup \{t\}$ and a subset $S \subseteq N$ with $|S| \geq k$ with respect to some row constraint $a^\transp x \leq b$ of \eqref{eq:knapsack_relaxation}, such that $(|S| - k + 1) x^*_t + \sum_{i \in S} x^*_i > |S|$? \\


\noindent\textbf{Problem WI-SP} \\
\textbf{Input: } $(A, d, c) \in (\Z^{m \times n}_+, \Z^m_+, \Z^n_+)$ and an optimal solution $x^*$ to the LP relaxation \eqref{eq:knapsack_relaxation}.\\
\textbf{Question: } Is there a pack $P$ with respect to some row constraint $a^\transp x \leq b$ of \eqref{eq:knapsack_relaxation}, such that $\sum_{i \in P} a_i x^*_i + \sum_{j \notin P} \max\{a_j-r(P),0\} x^*_j > \sum_{i \in P} a_i$? \\

For CI-SP, we have the following classic results.
\begin{theorem}[\cite{klabjan1998complexity}]
\label{theo:dsp_hard}
CI-SP is \NP-complete, even if $m=1$, or if $x^*$ is an extreme point.
\end{theorem}
We will show that the other three problems, ECI-SP, CONFIG-SP, and WI-SP are all \NP-complete.

Clearly, the \NP-hardness of the above problems imply the \NP-hardness of the more general separation problem where $x^*$ is a feasible, and not necessarily optimal, solution to \eqref{eq:knapsack_relaxation}. We should also remark that, since verifying if a given point violates a given inequality can be obviously done in polynomial time with respect to the input size of such point and inequality, the separation problems for these families of cuts are clearly in the class \NP. 
Therefore, when we discuss the separation complexity for those families of cuts, the concepts of \NP-hardness and \NP-completeness coincide.

This paper differs from the preliminary IPCO version~\cite{10.1007/978-3-031-06901-7_13} in the following aspects. 
Firstly, we solve the conjecture presented in Section~3 of \cite{10.1007/978-3-031-06901-7_13}. 
Secondly, we propose a new family of cutting-planes for the knapsack polytope, which can be seen as a natural extension of the $(1,k)$-configuration inequalities.
For this new family of inequalities, we obtain similar separation hardness result and polynomially-solvable cases.

\paragraph{Notation.} For an integer $n$, we set $[n] := \{1, 2, \ldots, n\}$.  
We define $e_n$ as the $n$-dimensional vector of ones, where we often repress the $n$ if the dimension of the vector is clear from the context.
For a vector $x \in \R^n$ and $S \subseteq [n]$, we set $x(S): = \sum_{i \in S} x_i$. 
So for a vector $a \in \R^n$, $a([n]) = a^\transp e = \sum_{i=1}^n a_i$. 
For a set $S \subseteq [n]$ and $k \leq |S|,$ we denote $S^{[k]}$ to be the set of the $k$ largest elements in $S$, and $S_{[k]}$ to be the set of the $k$ smallest elements in $S$. 
Here, when $k=0$, the above two sets are defined as the empty set.


\section{Extended Cover Inequalities}

In this section, we establish the complexity of the separation problem for extended cover inequalities, with a simple reduction from the separation problem for cover inequalities.
In the case where the point to be separated has a small number of fractional components, then extended cover inequality separation can be accomplished in polynomial time.

\begin{theorem}
\label{theo:ECI_reduction}
Problem ECI-SP is \NP-complete, even if $m=1$, or if $x^*$ is an extreme point solution to the LP relaxation~\eqref{eq:knapsack_relaxation}.
\end{theorem}

\begin{proof}
We transform CI-SP to ECI-SP.  
Let $(A,d,c,x^*) \in (\Z^{m \times n}_+, \Z^m_+,\Z^n_+, [0,1]^n)$ be the input to CI-SP.  
We construct the input to ECI-SP with the property that there is a yes-certificate to CI-SP with input $(A,d,c,x^*)$ if and only if there is a yes-certificate to ECI-SP with input $(A', d', c', y^*) \in (\Z^{m \times (n+1)}_+, \Z^m_+, \Z^{n+1}_+,[0,1]^{n+1})$.

The data for the ECI-SP instance are constructed as follows:
\begin{alignat*}{3}
A'_{ij} = A_{ij} \ \forall i \in [m], \forall j \in [n], &\qquad A'_{i,n+1} = \sum_{j=1}^n A_{ij} \ \forall i \in [m],\\
c'_j = c_j \ \forall j \in [n], &\qquad c'_{n+1} = M,\\
d'_i = d_i + \sum_{j=1}^n A_{ij} \ \forall i \in [m].
\end{alignat*}
The constant $M$ is chosen to be large enough so that if $x^*$ is an optimal solution to the linear program \eqref{eq:knapsack_relaxation}, then $y^* = (x^*,1)$ is an optimal solution to the linear program
\begin{equation}
\label{eq:lp_relaxation_eci}
\max\{(c')^\transp y : A' y \leq d', y \in [0,1]^{n+1}\}.
\end{equation}
It is a consequence of linear programming duality that selecting $M \geq (\pi^*)^\transp Ae$, where $\pi^*$ are optimal dual multipliers for the inequality constraints in \eqref{eq:knapsack_relaxation}, will ensure the optimality of $y^*$. Since there is an optimal solution $\pi^*$ whose encoding length is of polynomial size \cite{Schrijver86}, the encoding size of $M$ is a polynomial function of the input size of CI-SP.

Let $C \subseteq [n]$ be a cover with respect to a row constraint $a^\transp x \leq b$ of $Ax \leq d$
such that the associated CI does not hold at $x^*$, so $x^*(C) > |C| - 1$. 
Then $C': = C \cup \{n+1\}$ is a cover with respect to the constraint $(a^\transp, a^\transp e) \cdot y \leq b + a^\transp e$ within $A' y \leq d'$, and the associated ECI cuts off $y^*$, since $y^*(C') = 1+x^*(C) > |C| = |C'| - 1$.

On the other hand, assume that $C'$ is a cover with respect to some row constraint 
$a'^\transp y = (a^\transp, a^\transp e) \cdot y \leq b + a^\transp e = b'$ within $A' y \leq d'$ such that the associated ECI cuts off $y^*$.
Note that if $n+1 \notin C'$, then $\sum_{j \in C'} a'_j \leq a'([n]) = a^\transp e < b + a^\transp e$, and $C'$ cannot be a cover with respect to that row constraint. Thus, $n+1 \in C'$, and the ECI of $C'$ is just its cover inequality $y(C') \leq |C'| - 1$.
By construction, the set $C := C' \setminus \{n+1\}$ is a cover with respect to the constraint $a^\transp x \leq b$ within $Ax \leq d$.  
The ECI of $C'$ cuts off $y^*$, $y^*(C') = 1 + x^*(C) > |C'|-1 = |C|$, so $x^*(C) > |C| - 1$, and the CI from $C$ cuts off $x^*$.  

We have shown that there is a yes-certificate to CI-SP with input $(A,d,c,x^*)$ if and only if there is a yes-certificate to ECI-SP with input $(A', d', c', y^*)$.
Together with Theorem~\ref{theo:dsp_hard}, this establishes that ECI-SP is \NP-complete, even if $m=1$, or if $x^*$ is an extreme point to the LP relaxation~\eqref{eq:knapsack_relaxation}.  
To conclude the proof, it suffices to realize that the input $y^* = (x^*,1)$ for ECI-SP will be an extreme point of $\{y \in [0,1]^{n+1} : A'y \leq d'\}$ if $x^*$ is an extreme point of $\{x \in [0,1]^{n} : A x \leq d\}$.
\qed
\end{proof}

In the next theorem, we show that, if the fractional support of the input vector $x^*$ is ``sparse'', then we can separate ECIs in polynomial time. 
Here, the fractional support of $x^*$ denotes the set of index $i$ with $x^*_i \in (0,1)$. 

\begin{theorem}
\label{theo:poly_solvable_eci}
Let $x^*$ be the input solution to ECI-SP. If $|\{i \in [n]: x^*_i \in (0,1)\}| \leq \alpha$, then a separating ECI can be obtained in $O(n^2 \cdot 2^\alpha \cdot \log n \cdot m)$ time, if one exists.
\end{theorem}

\begin{proof}
For a given point $x^*$ and constraint $a^\transp x \leq b$ of $Ax \leq d$, there exists a separating ECI from the constraint if and only if for some $t \in [n]$, there exists a cover $C$ with $\max_{i \in C} a_i = a_t$, such that
\begin{equation}
\sum_{i \in [n]: a_i \geq a_t} x^*_i + \sum_{i \in C: a_i < a_t} x^*_i > |C|-1. \label{eci:1} 
\end{equation}
We partition $C$ into four sets, $C=T_1 \cup T_f \cup T_0 \cup T$, with $T_1 = \{i \in C : a_i < a_t, x^*_i = 1\}$, $T_f = \{i \in C : a_i < a_t, x^*_i \in (0,1)\}$, $T_0 = \{i \in C : a_i < a_t, x^*_i = 0\},$ and $T=\{i \in C : a_i = a_t\}$. 
With this definition, \eqref{eci:1} can be equivalently stated as
\begin{equation}
\label{eci:2}
    \sum_{i \in [n]: a_i \geq a_t} x^*_i > \sum_{i \in T_f} (1-x^*_i) + |T_0| + |T| - 1.
\end{equation}
The algorithm loops over all $t \in [n]$ and enumerates all $T_f \subseteq \{i \in [n] : a_i < a_t, x^*_i \in (0,1)\}$.  
By our  assumption on the cardinality of fractional support of $x^*$, there are $O(n \cdot 2^\alpha)$ iterations. 
For a fixed $t \in [n]$ and $T_f \subseteq C$, the separation problem then amounts to completing the cover $C$ so that
\begin{equation}
    \label{eci:2a}
    |T_0| + |T| < \sum_{i \in [n]: a_i \geq a_t} x_i^* - \sum_{i \in T_f} (1-x_i^*) + 1.
\end{equation}
The right-hand side of \eqref{eci:2a} is a constant, so separation for a fixed index $t$ and subset $T_f$ amounts to solving the knapsack problem
\begin{equation}
\label{eci:3}
\min_{z \in \{0,1\}^{|S_t|}} \left\{ \sum_{i \in S_t} z_i : \sum_{i \in S_t} a_i z_i \geq b_{t, T_f} \right\},
\end{equation}
where $S_t = \{i \in [n] : a_i = a_t \text{ or } a_i < a_t, x^*_i = 0\}$, and $b_{t, T_f} = b+1 - \sum_{i \in T_f} a_i - \sum_{i: a_i < a_t, x^*_i = 1} a_i.$ 
As the non-zero objective coefficients of the knapsack problem~\eqref{eci:3} are all the same, the problem can be solved by a simple greedy procedure, after sorting $\{a_i\}_{i \in S_t}$ in $O(n \log n)$ time. Therefore, overall this algorithm runs in $O(n^2 \cdot 2^\alpha \cdot \log n)$ time, since there are $O(n \cdot 2^\alpha)$ iterations. The theorem then follows by arguing for all $m$ constraints. 
\qed
\end{proof}

Theorem~\ref{theo:poly_solvable_eci} immediately implies the following corollary.

\begin{corollary}
\label{cor: 1}
If $m$ is polylogarithmic in $n$ and $x^*$ is an extreme point solution to \eqref{eq:knapsack_relaxation}, then a separating ECI can be obtained in polynomial time, if one exists. 
\end{corollary}

\begin{proof}
Since $x^*$ is an extreme point, we know that at most $m$ components of $x^*$ are fractional. Then Theorem~\ref{theo:poly_solvable_eci} implies that ECI-SP can be solved in $O(n^2 \cdot 2^m \cdot \log n \cdot m)$ time, which is polynomial in $n$ since $m$ is polylogarithmic in $n$. 
\qed
\end{proof}

\section{$(1,k)$-Configuration Inequalities}

In this section, we establish that the separation problem for $(1,k)$-configuration inequalities is \NP-complete using a reduction similar to the one in the proof of Theorem~\ref{theo:ECI_reduction}.

\begin{theorem}
\label{theo:CONFIG_reduction}
Problem CONFIG-SP is \NP-complete, even if $m=1$, or if $x^*$ is an extreme point solution to the LP relaxation~\eqref{eq:knapsack_relaxation}.
\end{theorem}

\begin{proof}
We transform CI-SP to CONFIG-SP. 
Given an input $(A,d,c,x^*)$ to CI-SP which is in $ (\Z^{m \times n}_+, \Z^m_+, \Z^n_+, [0,1]^n)$, we will construct a corresponding input $(A', d', c', y^*)$ to CONFIG-SP such that there is a yes-certificate to CI-SP with input $(A,d,c,x^*)$ if and only if there is a yes-certificate to CONFIG-SP with input $(A', d', c', y^*)$.

In the construction, the first $n$ columns of $A'$ are the columns of $A$, while the last two columns of $A'$ are the sum of all other columns; the right-hand side vector in the construction is $d' = d + 2 Ae$; and the first $n$ components of $c'$ are the same as those of $c$, while the last two components are a large positive constant:
\begin{alignat*}{3}
A'_{ij} = A_{ij} \ \forall i \in [m], \forall j \in [n], &\qquad A'_{ik} = \sum_{j=1}^n A_{ij} \ \forall i \in [m], \forall k \in \{n+1,n+2\},\\
c'_j = c_j \ \forall j \in [n], &\qquad c'_{n+1} = c'_{n+2} = M,\\
d'_i = d_i + 2 \sum_{j=1}^n A_{ij} \ \forall i \in [m].
\end{alignat*}
The constant $M$ is chosen to be large enough so that if $x^*$ is an optimal solution to the linear program \eqref{eq:knapsack_relaxation}, then $y^* = (x^*,1,1)$ is an optimal solution to the linear program
\begin{equation}
\label{eq:lp_relaxation_config}
\max\{(c')^\transp y : A' y \leq d', y \in [0,1]^{n+2}\}.
\end{equation}
It is a consequence of linear programming duality that selecting $M \geq (\pi^*)^\transp Ae$, where $\pi^*$ are optimal dual multipliers for the inequality constraints in \eqref{eq:knapsack_relaxation}, will ensure the optimality of $y^*$. 
As it is know that there is an optimal solution $\pi^*$ whose encoding length is of polynomial size \cite{Schrijver86}, we see that $M$ exists and its encoding size is a polynomial function of the input size of CI-SP.

Let $C \subseteq [n]$ be a cover with respect to the row constraint $a^\transp x \leq b$ of $A x \leq d$ such that the associated CI cuts off $x^*: x^*(C) > |C| - 1$. Let $C': = C \cup \{n+1,n+2\}$. By definition of the input $(A', d')$, $C'$ is a cover with respect to the row constraint $a'^\transp y = (a^\transp , a^\transp e, a^\transp e) y \leq b + 2 a^\transp e = b'$
within $A' y \leq d'$, and the associated CI cuts off $y^*: y^*(C') = 2 + x^*(C) > |C| + 1 = |C'| - 1$.  As every cover inequality is dominated by a minimal cover inequality, there is a a minimal cover contained in $C'$ whose associated minimal CI cuts off $y^*$.
Every minimal CI is a special case of $(1,k)$-configuration inequality, so we have obtained a $(1,k)$-configuration inequality with respect to a row constraint within $A' y \leq d'$ that cuts off $y^*$.

To complete the proof, we must show that if $N \cup \{t\}$ is a $(1,k)$-configuration with respect to some row constraint $a'^\transp y \leq b'$ within $A' y \leq d'$, and $S \subseteq N$ with $|S| \geq k$, such that
$(|S| - k + 1) y^*_t + \sum_{i \in S} y^*_i > |S|$,
then we can construct a cover $C$ with respect to the associated row constraint in $Ax \leq b$ such that the associated CI cuts off $x^*$.

By construction of $(A', d')$, we know that the row constraint $a'^\transp y \leq b'$ takes the form $(a^\transp , a^\transp e, a^\transp e) y \leq b + 2a^\transp e$, where $a^\transp x \leq b$ is a row constraint in $A x \leq b$. 
First, observe that in the constraint $(a^\transp , a^\transp e, a^\transp e) y \leq b + 2a^\transp e$, any cover must contain both $n+1$ and $n+2$.  
By definition of a $(1,k)$-configuration, for any subset $Q \subseteq N$ with $|Q| = k$, $Q \cup \{t\}$ is a minimal cover. Specifically, $N \cup \{t\}$ is a cover. 
This implies that $\{n+1,n+2\} \subseteq N \cup \{t\}$, which means $N$ must contain $n+1$ or $n+2$. If $k \leq |N| - 1$, then for any $i' \in N$, the set $N \cup \{t\} \setminus \{i'\}$ will also be a cover. 
However, since $N$ contains $n+1$ or $n+2$, then when $i' = n+1$ (or $n+2$), $N \cup \{t\} \setminus \{i'\}$ will not be a cover. 
Therefore, $|N| = k$, and the associated $(1,k)$-configuration inequality reduces to a minimal CI, so
\begin{equation}
\label{eq:1k_ci}
y^*(N \cup \{t\}) > |N|.
\end{equation}
Let $C = N \cup \{t\} \setminus \{n+1,n+2\}$, so $|C| = |N| - 1$. Since $N \cup \{t\}$ is a cover with respect to the constraint
$(a^\transp , a^\transp e, a^\transp e) y \leq b + 2a^\transp e$, and $n+1,n+2 \in N \cup \{t\}$, we can infer that $a(C) + 2 a^\transp e > b+2 a^\transp e$, so $C$ is a cover with respect to the row constraint $a^\transp x \leq b$ of $Ax \leq d$.  Furthermore, from \eqref{eq:1k_ci} and the definition of $y^* = (1,1,x^*)$, we have $x^*(C) > |N| - 2 = |C| - 1$. Therefore, we end up with a cover $C$ whose associated CI cuts off $x^*$. 

We have thereby shown that there is a yes-certificate to CI-SP with input $(A,d,c,x^*)$ if and only if there is a yes-certificate to CONFIG-SP with input $(A', d', c', y^*)$.
Together with Theorem~\ref{theo:dsp_hard}, our proof establishes that CONFIG-SP is \NP-complete, even if $m=1$, or if $x^*$ is an extreme point to the LP relaxation~\eqref{eq:knapsack_relaxation}. 
To conclude the proof, it suffices to realize that the input $y^* = (x^*,1, 1)$ for CONFIG-SP will be an extreme point of $\{y \in [0,1]^{n+2} : A'y \leq d'\}$ if $x^*$ is an extreme point of $\{x \in [0,1]^{n} : A x \leq d\}$.
\qed
\end{proof}

We have settled the complexity of the separation problem for $(1,k)$-configuration inequalities, for an input solution $x^*$ that is an extreme point to the LP-relaxation~\eqref{eq:knapsack_relaxation}. 
In our preliminary IPCO version \cite{10.1007/978-3-031-06901-7_13}, we conjectured the separation problem to be \NP-complete, even for points $x^*$ with a small number of fractional components. 
However, we are able to refute this conjecture.
We show that, in this special case, the separation problem of $(1,k)$-configuration inequalities is in fact polynomially solvable. 
The following easy lemma gives an equivalent condition for a set $N \cup \{t\}$ to be a $(1,k)$-configuration.

\begin{lemma}
\label{lem: equi_config}
For knapsack constraint $a^\transp x \leq b$ with $a_1 \leq a_2 \leq \ldots \leq a_n$, set $N \cup \{t\} \subseteq [n]$ is a $(1,k)$-configuration with respect to the knapsack constraint if and only if $a(N) \leq b$, $a_t + a(N^{[k-1]}) \leq b$, and $a_t + a(N_{[k]}) > b$. 
\end{lemma}

Recall that, in the statement of Lemma~\ref{lem: equi_config}, $N^{[k-1]}$ denotes the set of the $k-1$ largest elements in $N$ and $N_{[k]}$ denotes the set of the $k$ smallest elements in $N$.

\begin{proof}
By definition of $(1,k)$-configuration, it suffices to check that, when $a(N) \leq b$, $Q \cup \{t\}$ is a minimal cover for every $Q \subseteq N$ with $|Q| = k$ if and only if  $a_t + a(N^{[k-1]}) \leq b$ and $a_t + a(N_{[k]}) > b$.

Picking $Q = N_{[k]}$, $Q \cup \{t\}$ being a cover implies that $a_t + a(N_{[k]}) > b$. Picking $Q = N^{[k]}$, $Q \cup \{t\}$ being a minimal cover implies that $a_t + a(N^{[k-1]}) \leq b$. 
Next we want to show that, for any $Q \subseteq N$ with $|Q| = k$, if $a(N) \leq b$, $a_t + a(N^{[k-1]}) \leq b$, and $a_t + a(N_{[k]}) > b$, then $Q \cup \{t\}$ is a minimal cover. 
By assumption $a_1 \leq a_2 \leq \ldots \leq a_n$, we know that $a(Q) \geq a(N_{[k]})$. 
So $a(Q) + a_t \geq a(N_{[k]}) + a_t > b$ implies that $Q \cup \{t\}$ is a cover. Now arbitrarily pick $i' \in Q \cup \{t\}$. 
If $i' = t$, then $a(Q) \leq a(N) \leq b$; If $i' \in Q$, then $a(Q \setminus \{i'\}) \leq a(N^{[k-1]})$, which gives $a(Q \setminus \{i'\}) + a_t \leq a_t + a(N^{[k-1]}) \leq b$. 
Hence, $Q \cup \{t\}$ is a minimal cover. 
\qed
\end{proof}

\begin{theorem}
\label{theo:poly_solvable_config}
Let $x^*$ be the input solution to CONFIG-SP. If $|\{i \in [n]: x^*_i \in (0,1)\}|$ is bounded by a constant, then a separating $(1,k)$-configuration inequality can be obtained in polynomial time, if one exists.
\end{theorem}

\begin{proof}
Throughout the proof, we assume that $|\{i \in [n] : x^*_i \in (0,1)\}| \leq \alpha$, for some constant $\alpha$. 
Without loss of generality, we further assume that there does not exist a separating cover inequality for $x^*$. Since if there exists a separating cover inequality, then one can easily find it in polynomial time (see, e.g., Theorem~2 in \cite{klabjan1998complexity}), which is a special case of $(1,k)$-configuration inequality. 
Now, we assume that $\sum_{i \in S}x_i + (|S|-k+1)x_t \leq |S|$ is a $(1,k)$-configuration inequality from $(1,k)$-configuration $N \cup \{t\}$ with respect to knapsack constraint $a^\transp x \leq b$,
that is violated by point $x^*$. 
We want to show that one can find a separating $(1,k')$-configuration inequality in polynomial-time. 
Here, the separating inequality we find might be different from $\sum_{i \in S}x_i + (|S|-k+1)x_t \leq |S|$, and $k'$ might be different from $k$. Also, without loss of generality, here we assume the knapsack constraint $a^\transp x \leq b$ satisfies $a_1 \leq a_2 \leq \ldots \leq a_n$.

Since $S \subseteq N$, $|S| \geq k$, and $N \cup \{t\}$ is a $(1,k)$-configuration, by definition, $S \cup \{t\}$ is also a $(1,k)$-configuration with respect to constraint $a^\transp x \leq b$. 
First, we have the following claim.
\begin{claim}
There exists $S^* \subseteq S$, such that $S^* \cup \{t\}$ is a $(1,k)$-configuration whose corresponding $(1,k)$-configuration inequality separates $x^*$, and $x^*_i > 0$ for any $i \in S^* \cup \{t\}$.
\end{claim}
\begin{cpf}
From $\sum_{i \in S}x^*_i + (|S|-k+1)x^*_t > |S|$, we know that $x^*_t > 0$. If for some $i' \in S$, there is $x^*_{i'} = 0$, then consider $S' = S \setminus \{i'\}$. Here $|S'| \geq k$. If not, then $|S| = k$, which means the original $(1,k)$-configuration inequality is simply a cover inequality. From $\sum_{i \in S}x^*_i + x^*_t > |S|$, we know that $x^*_{i'} = 0$ is impossible. 
Here $S' \cup \{t\}$ is also a $(1,k)$-configuration. Moreover:
$$
\sum_{i \in S'}x^*_i + (|S'|-k+1)x^*_t = \sum_{i \in S}x^*_i + (|S|-k+1)x^*_t - x^*_t > |S| - 1 = |S'|.
$$
Therefore, $S' \cup \{t\}$ gives another separating $(1,k)$-configuration inequality. 
Recursively removing index $i \in S'$ with $x^*_i = 0$, we end up obtaining a set $S^* \subseteq S$ with the desired properties.
\end{cpf}

Due to the above claim, without loss of generality, we will assume $x^*_i > 0$ for any $i \in S \cup \{t\}$.
Denote $\bar{S}:=S^{[k-1]} \setminus S_{[k]}$, $\underline{S}:= S_{[k]} \setminus S^{[k-1]}$, $T = S \setminus (\bar{S} \cup \underline{S})$, and $N_1:= \{i \in [n] \setminus \{t\} : x^*_i = 1\}$. 
Throughout, we assume that $\bar{S} \neq \emptyset$, since $\bar{S}= \emptyset$ corresponds to the case of $|S| = k$, which means the original separating $(1,k)$-configuration inequality is simply a cover inequality.
Since $x^*$ is a feasible solution to the LP relaxation \eqref{eq:knapsack_relaxation}, we know $a(N_1) \leq b$. Denote $\Delta: = b - a(N_1)$. 
\begin{claim}
\label{claim: Delta}
$a_t \leq \Delta$.
\end{claim}
\begin{cpf}
Consider set $N_1 \cup \{t\}$. 
If $N_1 \cup \{t\}$ is a cover to knapsack constraint $a^\transp x \leq b$, then the cover inequality $\sum_{i \in N_1} x_i + x_t \leq |N_1|$ will be violated by $x^*$, because $x^*_i = 1$ for any $i \in N_1$ and $x^*_t > 0$. 
This contradicts our initial assumption. 
Thus, $a(N_1) + a_t \leq b = a(N_1) + \Delta$, which yields that $a_t \leq \Delta$.
\end{cpf}
\begin{claim}
\label{claim: bound_S}
$|\bar{S}| < 2\alpha$ and $|\underline{S}| < 2\alpha + 1$. 
\end{claim}
\begin{cpf}
First, let $S_{[k], 1}: = \{i \in S_{[k]} : x^*_i = 1\}$, $S_{[k], f}: = \{i \in S_{[k]} : x^*_i \in (0,1)\}$ denote the binary support and fractional support of $S_{[k]}$, let $S^{[k-1]}_1: = \{i \in S^{[k-1]} : x^*_i = 1\}$, $S^{[k-1]}_f: = \{i \in S^{[k-1]} : x^*_i \in (0,1)\}$ denote the binary support and fractional support of $S^{[k-1]}$, and let $\bar{S}_1: = \{i \in \bar S : x^*_i = 1\}$, $\bar{S}_f: = \{i \in \bar S : x^*_i \in (0,1)\}$ denote the binary support and fractional support of $x^*$ over $\bar S$. 
By Lemma~\ref{lem: equi_config}, we have 
$$a_t + a(S_{[k], 1}) + a(S_{[k], f}) = a_t + a(S_{[k]}) > b = a(N_1) + \Delta.$$
Hence, we obtain
\begin{equation}
\label{eq: 1}
a(N_1 \setminus S_{[k], 1}) < a(S_{[k], f}) + a_t - \Delta.
\end{equation}
From Claim~\ref{claim: Delta} and \eqref{eq: 1}, we obtain that $a(N_1 \setminus S_{[k], 1}) < a(S_{[k], f})$. 
Notice that $\bar{S}_1 \subseteq N_1$ and $\bar{S}_1 \cap S_{[k], 1} = \emptyset$, so $\bar{S}_1 \subseteq N_1 \setminus S_{[k], 1}$ and
\begin{equation}
\label{eq: 2}
a(\bar{S}_1) \leq a(N_1 \setminus S_{[k], 1}) < a(S_{[k], f}). 
\end{equation}
Here by definition, $\bar{S} \subseteq \bar{S} =S^{[k-1]} \setminus S_{[k]}$. 
Since $S^{[k-1]}$ denotes the $k-1$ largest elements in $S$, $S_{[k]}$ denotes the $k$ smallest elements in $S$, and $a_1 \leq \ldots \leq a_n$, we know that for any $i \in \bar{S}_1$ and $j \in S_{[k], f}$, we have $a_i \geq a_j$. In this case, \eqref{eq: 2} implies that $|\bar{S}_1| < |S_{[k], f}| \leq \alpha$.
Here, the inequality $|S_{[k], f}| \leq \alpha$ holds because of our initial assumption that $|\{i \in [n]: x^*_i \in (0,1)\}| \leq \alpha$. 
Therefore, 
$$
|\bar{S}| = |\bar{S}_1| + |\bar{S}_f| < 2 \alpha.
$$
Here the last inequality also utilizes the fact that $|\bar{S}_f| \leq |\{i \in [n]: x^*_i \in (0,1)\}| \leq \alpha$. Since $|\underline{S}| = |\bar S| + 1$, we also obtain that $|\underline{S}| < 2\alpha + 1$. 
\end{cpf}
Next, we continue our discussion by considering separately two cases, depending on whether $S^{[k-1]} \cap S_{[k]} = \emptyset$ or not. 
Here, we also decompose $T = T_f \cup T_1,$ where $T_1: = \{i \in T : x^*_i = 1\}$, $T_f: = \{i \in T : x^*_i \in (0,1)\}$. 
We remark that the size of the fractional support $T_f$ of $T$ satisfies $|T_f| \leq \alpha$.

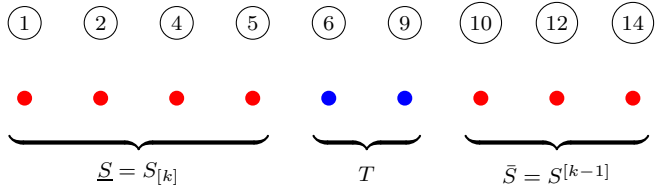
\begin{figure}[t]
\begin{center}
\begin{tikzpicture}[ultra thick]
\node at (0,1) {\circled{1}};
\node at (1,1) {\circled{2}};
\node at (2,1) {\circled{4}};
\node at (3,1) {\circled{5}};
\node at (4,1) {\circled{6}};
\node at (5,1) {\circled{9}};
\node at (6,1) {\circled{10}};
\node at (7,1) {\circled{12}};
\node at (8,1) {\circled{14}};
\filldraw[red] (0,0) circle (2pt);
\filldraw[red] (1,0) circle (2pt);
\filldraw[red] (2,0) circle (2pt);
\filldraw[red] (3,0) circle (2pt);
\filldraw[blue] (4,0) circle (2pt);
\filldraw[blue] (5,0) circle (2pt);
\filldraw[red] (6,0) circle (2pt);
\filldraw[red] (7,0) circle (2pt);
\filldraw[red] (8,0) circle (2pt);
\draw [decorate, decoration = {calligraphic brace, mirror, amplitude=5pt}] (-0.2,-0.5) -- (3.2,-0.5);
\node at (1.5, -1) {$\underline{S} = S_{[k]}$}; 
\draw [decorate, decoration = {calligraphic brace, mirror, amplitude=5pt}] (5.8,-0.5) -- (8.2,-0.5);
\node at (7, -1) {$\bar{S} = S^{[k-1]}$}; 
 \draw [decorate, decoration = {calligraphic brace, mirror, amplitude=5pt}] (3.8,-0.5) -- (5.2,-0.5);
\node at (4.5, -1) {$T$}; 
\end{tikzpicture}
\end{center}
\caption{An example for Case~1: $S=\{1,2,4,5,6,9,10,12,14\}$ and $k=4$. Here, $S_{[k]} = \{1,2,4,5\}$, $S^{[k-1]} = \{10,12,14\}$, and $S_{[k]} \cap S^{[k-1]} = \emptyset$.}
\label{fig: case1}
\end{figure}

\smallskip \noindent
\textbf{Case 1:} $S^{[k-1]} \cap S_{[k]} = \emptyset$. 
In this case, we know $S_{[k]} = \underline{S}$, $S^{[k-1]} = \bar{S}$. 
See Fig.~\ref{fig: case1} for an example. 
By Claim~\ref{claim: bound_S}, we have $k = |\underline{S}| < 2\alpha+1$. 
By Lemma~\ref{lem: equi_config}, in order for $S \cup \{t\}$ to be a $(1,k)$-configuration, we must have: $a(\underline{S}) + a_t > b$, $a(\bar{S}) + a_t \leq b$, and $a(\underline{S}) + a(\bar{S}) + a(T) = a(S) \leq b$. 
Then, for fixed $t, \underline{S}, \bar{S}$ and $T_f$, set $T_1$ satisfies the following:
\begin{align}
    & a(T_1) \leq b-a(\underline{S}) - a(\bar{S}) - a(T_f), \label{alig: 1}\\
    & |T_1| > (|\bar{S}| + |\underline{S}| + |T_f|) \frac{1-x^*_t}{x^*_t} + k-1 - \frac{x^*(\bar{S}) + x^*(\underline{S}) + x^*(T_f)}{x^*_t}, \label{alig: 2}\\
    & T_1 \subseteq \{i \in N_1 : \max_{j \in \underline{S}} j < i < \min_{j \in \bar{S}} j \}. \label{alig: 3}
\end{align}
Here, \eqref{alig: 3} holds because $T = S \setminus (\underline{S} \cup \bar{S})$, and \eqref{alig: 2} is derived from the fact that
$$
(|S|-k+1) x^*_t + \sum_{i \in S} x^*_i > |S|,
$$
and $|S| = |\underline{S}| + |T_1| + |T_f| + |\bar{S}|$.

Notice that for fixed $t, \underline{S}, \bar{S}$, and $T_f$, the right hand side constants of \eqref{alig: 1}
 and \eqref{alig: 2} are fixed as well, so the following greedy procedure will produce a feasible $T_1$ satisfying \eqref{alig: 1}--\eqref{alig: 3} if one exists: begin with $T_1 = \emptyset$, within set $\{i \in N_1 : \max_{j \in \underline{S}} j < i < \min_{j \in \bar{S}} j \}$, recursively add the smallest index $i$ (which corresponds to the smallest weight $a_i$) into set $T_1$, and stop when \eqref{alig: 1} becomes invalid after the next addition, or $T_1 = \{i \in N_1 : \max_{j \in \underline{S}} j < i < \min_{j \in \bar{S}} j \}$.

Therefore, in the case $S_{[k]} \cap S^{[k]} = \emptyset$, CONFIG-SP can be solved by the following procedure:
\begin{enumerate}
\item Arbitrarily pick $t \in [n]$, $k \leq 2\alpha$, and set $\underline{S}, \bar{S}, T_f$ such that: 
    \begin{align*}
    & |\underline{S}| = k = |\bar{S}|+1,\\
    & \max_{j \in \underline{S}} j < \min_{j \in \bar{S}} j,\\
    & a(\underline{S}) + a_t > b, \\
    & a(\bar{S}) + a_t \leq b, \\
    & a(\underline{S}) + a(\bar{S}) \leq b, \\
    & T_f \subseteq \{i \in [n]: x^*_i \in (0,1), \max_{j \in \underline{S}}j < i < \min_{j \in \bar{S}} j\}.
    \end{align*}
\item Begin with $T_1 = \emptyset$, within set $\{i \in N_1 : \max_{j \in S_{[k]}} j < i < \min_{j \in S^{[k-1]}} j \}$, recursively add the smallest index into set $T_1$, and stop when \eqref{alig: 1} becomes invalid after the next addition, or $T_1 = \{i \in N_1 : \max_{j \in \underline{S}} j < i < \min_{j \in \bar{S}} j \}$.
\item If \eqref{alig: 2} holds, then the set $\underline{S}\cup T_1 \cup T_f \cup \bar{S} \cup \{t\}$ is a $(1,k)$-configuration whose associated $(1,k)$-configuration inequality is violated by $x^*$. 
\end{enumerate}
Note that step~2 can be performed in $O(n \log n)$ time, step~1 can be performed in $O(n)$ time, and the total number of combination of different $t \in [n]$, $k \leq 2\alpha$, and set $\underline{S}, \bar{S}, T_f$ with $|\underline{S}|= k = |\bar{S}|+1$ can be upper bounded by:
$$
n \cdot 2\alpha \cdot \sum_{k=0}^{2\alpha} \binom{n}{k} \cdot \sum_{k=0}^{2\alpha-1} \binom{n}{k} \cdot 2^\alpha = O(n^{4\alpha} \cdot 2^{\alpha+1} \cdot \alpha).
$$
Hence, the above procedure can be implemented in $O(n^{4\alpha+1} \cdot 2^{\alpha+1} \cdot \log n \cdot \alpha)$ time. 

\begin{figure}[t]
\begin{center}
\begin{tikzpicture}[ultra thick]
\node at (0,2) {\circled{1}};
\node at (1,2) {\circled{2}};
\node at (2,2) {\circled{4}};
\node at (3,2) {\circled{5}};
\node at (4,2) {\circled{6}};
\node at (5,2) {\circled{9}};
\node at (6,2) {\circled{10}};
\node at (7,2) {\circled{12}};
\node at (8,2) {\circled{14}};
\filldraw[red] (0,0) circle (2pt);
\filldraw[red] (1,0) circle (2pt);
\filldraw[red] (2,0) circle (2pt);
\filldraw[red] (3,0) circle (2pt);
\filldraw[blue] (4,0) circle (2pt);
\filldraw[blue] (5,0) circle (2pt);
\filldraw[red] (6,0) circle (2pt);
\filldraw[red] (7,0) circle (2pt);
\filldraw[red] (8,0) circle (2pt);
\draw [decorate, decoration = {calligraphic brace, mirror, amplitude=5pt}] (-0.2,-0.5) -- (3.2,-0.5);
\node at (1.5, -1) {$\underline{S}$}; 
\draw [decorate, decoration = {calligraphic brace, mirror, amplitude=5pt}] (5.8,-0.5) -- (8.2,-0.5);
\node at (7, -1) {$\bar{S}$}; 
 \draw [decorate, decoration = {calligraphic brace, mirror, amplitude=5pt}] (3.8,-0.5) -- (5.2,-0.5);
\node at (4.5, -1) {$T$}; 
\draw [pen colour={gray}, decorate, decoration = {calligraphic brace}] (-0.2,0.5) -- (5.2,0.5);
\node at (2.5, 0.9) {$S_{[k]}$}; 
\draw [pen colour={gray}, decorate, decoration = {calligraphic brace}] (3.8,0.8) -- (8.2,0.8);
\node at (6, 1.2) {$S^{[k-1]}$};
\end{tikzpicture}
\end{center}
\caption{An example for Case~2. $S=\{1,2,4,5,6,9,10,12,14\}$ and $k=6$. Here, $S_{[k]} = \{1,2,4,5,6,9\}$, $S^{[k-1]} = \{6,9,10,12,14\}$, and $S_{[k]} \cap S^{[k-1]} = T \neq \emptyset$.}
\label{fig: case2}
\end{figure}
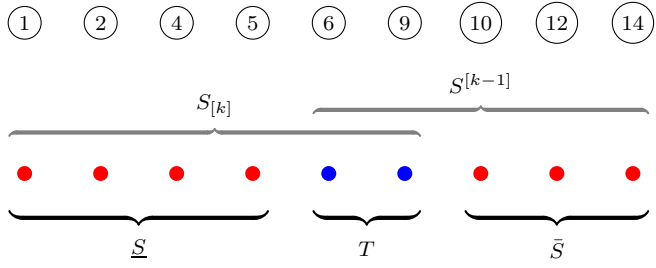

\smallskip \noindent
\textbf{Case 2:} $S^{[k-1]} \cap S_{[k]} \neq \emptyset$. In this case, $T = S \setminus (\bar{S} \cup \underline{S}) = S^{[k-1]} \cap S_{[k]}$. See Fig.~\ref{fig: case2} for an example. For fixed $\underline{S}$ and $\bar{S}$, denote $D_1: = \{i \in N_1 : \max_{j \in \underline{S}} j < i < \min_{j \in \bar{S}} j \}$. Then clearly, $T_1 \subseteq D_1$. 
\begin{claim}
$|D_1 \setminus T_1| < 2\alpha$. 
\end{claim}
\begin{cpf}
From 
$$
a_t + a(S_{[k]})  = a_t + a(\underline{S}) + a(T_1) + a(T_f) > b = a(D_1) + b-a(D_1),
$$
we obtain that
\begin{equation}
\label{eq: 10}
    a(D_1 \setminus T_1) < a(T_f) + a(\underline{S}) + a_t - \left(b-a(D_1) \right).
\end{equation}
Let $\underline{S}_1:=\{i \in \underline{S}: x^*_i = 1\}$, $\underline{S}_f:= \{i \in \underline{S}:x^*_i \in (0,1)\}$, $N_f:=\{i \in [n] \setminus \{t\}: x^*_i \in (0,1)\}$. 
Since $a^\transp x^* \leq b$, we have
\begin{equation}
\label{eq: 11}
a(\underline{S}_1) + \sum_{i \in N_f}a_i x^*_i + a_t x^*_t + a(D_1) \leq a^\transp x^* \leq b.
\end{equation}
Moreover, from Claim~\ref{claim: Delta} and $\Delta = b-a(N_1) \leq \sum_{i \in N_f} a_i x^*_i + a_t x^*_t$, we have 
\begin{equation}
\label{eq: 12}
 a_t \leq \sum_{i \in N_f} a_i x^*_i + a_t x^*_t.    
\end{equation}
Therefore,
\begin{align*}
    a(\underline{S}) + a_t & = a(\underline{S}_1) + a(\underline{S}_f) + a_t && \\
    & \leq a(\underline{S}_1) + a(\underline{S}_f) + \sum_{i \in N_f} a_i x^*_i + a_t x^*_t && (\text{by } \eqref{eq: 12}) \\
    & \leq  a(\underline{S}_f) + b-a(D_1). && (\text{by } \eqref{eq: 11})
\end{align*}
Combined with \eqref{eq: 10}, we obtain
\begin{equation}
\label{eq: 13}
a(D_1 \setminus T_1) < a(T_f) + a(\underline{S}_f).    
\end{equation}
Since $a_t + a(\underline{S}) + a(T) = a_t + a(S_{[k]}) > b \geq a_t + a(S^{[k-1]}) = a_t + a(\bar{S}) + a(T)$, we have $a(\underline{S}) > a(\bar{S})$. Let $i^* := \max_{j \in \underline{S}} j$ and $j^*:=\min_{j \in \bar{S}} j$. 
Then we have
$$
|\underline{S}| \cdot a_{i^*} \geq a(\underline{S}) > a(\bar{S}) \geq |\bar{S}| \cdot a_{j^*}.
$$
Since $|\bar{S}| = |\underline{S}|-1 \geq 1$, we obtain $a_{j^*} < 2 a_{i^*}$. 
Together with \eqref{eq: 13}, we have
$$
|D_1 \setminus T_1| \cdot a_{i^*} \leq a(D_1 \setminus T_1) < a(T_f \cup \underline{S}_f) \leq |T_f \cup \underline{S}_f| \cdot a_{j^*} \leq \alpha \cdot a_{j^*} < 2 \alpha \cdot a_{i^*}.
$$
This concludes the proof of the claim.
\end{cpf}
From the above claim, we know that, for fixed $\underline{S}$ and $\bar{S}$, there can only be polynomially many sets $T_1 \subseteq D_1$.
Therefore, in the case of $S_{[k]} \cap S^{[k]} \neq \emptyset$, CONFIG-SP can be solved by the following procedure:
\begin{enumerate}
\item Arbitrarily pick $t \in [n]$ and set $\underline{S}, \bar{S}, T_f, T_1^c$, such that:
    \begin{align*}
    & |\underline{S}| = |\bar{S}|+1 \leq 2\alpha,\\
    & a(\underline{S}) + a(\bar{S}) \leq b, \\
    & \max_{j \in \underline{S}} j < \min_{j \in \bar{S}} j,\\
    & T_f \subseteq \{i \in [n]: x^*_i \in (0,1), \max_{j \in \underline{S}}j < i < \min_{j \in \bar{S}} j\}, \\
    & T_1^c \subseteq \{i \in [n]: x^*_i=1, \max_{j \in \underline{S}}j < i < \min_{j \in \bar{S}} j\},\\
    & |T_1^c| \leq 2\alpha-1.
    \end{align*}
\item 
Let $S=\underline{S} \cup T_f \cup (D_1 \setminus T_1^c) \cup \bar{S}$ and $k=|S|-|\bar{S}|$, where $D_1 = \{i \in N_1 : \max_{j \in \underline{S}} j < i < \min_{j \in \bar{S}} j \}$. 
If 
\begin{align*}
& a(S) \leq b, a_t + a(S \setminus \bar{S}) > b \geq a_t + a(S \setminus \underline{S}), \\
& (|S|-k+1)x^*_t + x^*(S) > |S|,
\end{align*}
then $S \cup \{t\}$ gives a separating $(1,k)$-configuration inequality for $x^*$. 
\end{enumerate}
Here, the total number of combinations of different $t \in [n]$ and sets $\underline{S}, \bar{S}, T_f, T_1^c \subseteq [n]$ with $|\underline{S}| = |\bar{S}|+1 \leq 2\alpha$, $ T_f \subseteq \{i \in [n]: x^*_i \in (0,1)\}$, and $|T_1^c| \leq 2\alpha-1$
can be upper bounded by
$$
n \cdot \sum_{k=0}^{2\alpha} \binom{n}{k} \cdot \sum_{k=0}^{2\alpha-1} \binom{n}{k} \cdot 2^\alpha \cdot \sum_{k=0}^{2\alpha-1} \binom{n}{k} = O(n^{6\alpha-1}\cdot 2^\alpha).
$$
Hence, the above procedure can be implemented in polynomial time. 

From the discussion provided in the above two cases, the proof is complete.
\qed
\end{proof}

Similarly to how Corollary~\ref{cor: 1} directly follows from Theorem~\ref{theo:poly_solvable_eci}, Theorem~\ref{theo:poly_solvable_config} implies the following corollary.

\begin{corollary}
\label{cor: config}
If $m$ is a constant and $x^*$ is an extreme point solution to \eqref{eq:knapsack_relaxation}, then a separating $(1,k)$-configuration inequality can be obtained in polynomial time, if one exists. 
\end{corollary}

\section{Generalized $(1,k)$-Configuration Inequalities}

In this section, we study a simple generalization of the $(1,k)$-configuration inequalities, which we call \emph{generalized $(1,k)$-configuration inequalities}.

\begin{definition}
For knapsack constraint $a^\transp x \leq b$, a set $N \cup \{t\}$ with $N \subsetneq [n]$ and $t \notin N$ is called a \emph{generalized $(1,k)$-configuration} for $k \leq |N|$, if $Q \cup \{t\}$ is a cover for every $Q \subseteq N$ with $|Q| = k$. For any generalized $(1,k)$-configuration $N \cup \{t\}$, the inequality 
\begin{equation}
    \label{eq: generalized_config}
    (n'-k+1) x_t + \sum_{i \in N} x_i \leq n'
\end{equation}
is called a \emph{generalized $(1,k)$-configuration inequality}. Here, $n'$ is the smallest number such that inequality $\sum_{i \in N} x_i \leq n'$ is valid for $K$. 
\end{definition}

It is easy to see that any $(1,k)$-configuration inequality is a generalized $(1,k)$-configuration inequality. 
In fact, when $a(N) \leq b$ and $Q \cup \{t\}$ is a \emph{minimal} cover for every $Q \subseteq N$ with $|Q| = k$, then the inequality \eqref{eq: generalized_config} becomes a $(1,k)$-configuration inequality. 
The next simple theorem states that any generalized $(1,k)$-configuration inequality is also valid for the knapsack polytope $K$. 

\begin{theorem}
Let $N \cup \{t\}$ be a generalized $(1,k)$-configuration with respect to knapsack constraint $a^\transp x \leq b$. Then inequality~\eqref{eq: generalized_config} is valid for $K$. 
\end{theorem}
\begin{proof}
Arbitrarily pick $x' \in \{0,1\}^n$ with $a^\transp x' \leq b$. 
We consider the component $x'_t$. 
If $x'_t = 0$, then \eqref{eq: generalized_config} reduces to $\sum_{i \in N} x'_i \leq n'$, which is true because of the assumption on $n'$. 
If $x'_t = 1$, then \eqref{eq: generalized_config} reduces to $\sum_{i \in N} x'_i \leq k-1$. If $\sum_{i \in N} x'_i \geq k$, then arbitrarily pick $Q \subseteq \{i \in N: x'_i = 1\}$ with $|Q| = k$. By definition of generalized $(1,k)$-configuration, we know that $Q \cup \{t\}$ is a cover. However, $a(Q \cup \{t\}) \leq a^\transp x' \leq b$, which gives the contradiction.
\qed
\end{proof}

The next two examples show that the class of generalized $(1,k)$-configuration inequalities is strictly broader than the class of $(1,k)$-configuration inequalities. 
In particular, by relaxing either the ``minimal cover'' assumption, or the assumption $a(N) \leq b$, in the definition of $(1,k)$-configuration inequality, one is able to obtain different facet-defining inequalities.

\begin{example}
Consider the knapsack constraint $(2,4,5,6,7,20)^\transp x \leq 30$. 
For $N=\{1,2,3,4,5\}$ and $k=3$, $N \cup \{6\}$ is a generalized $(1,k)$-configuration, but it is not a $(1,k)$-configuration. 
This is because $\{4,5,6\}$ is a cover, so $\{3,4,5\} \cup \{t\}$ cannot be a minimal cover. 
However, for any $Q \subseteq N$ with $|Q| = 3$, we have $a_6 + a(Q) \geq a_6 + a(N_{[3]}) = a_6 + a_1 + a_2 + a_3 = 31 > 30$. 
Furthermore, the corresponding generalized $(1,k)$-configuration inequality 
$\sum_{i=1}^5 x_i + 3x_6 \leq 5$ is facet-defining.
Similarly, another facet-defining inequality $x_1 + x_2 + x_4 + x_5 + 2x_6 \leq 4$ is a generalized $(1,k)$-configuration inequality from $N=\{1,2,4,5\}$, $t=6$, $k=3$, while it is not a $(1,k)$-configuration inequality. 
Notice that in both cases, the $(1,k)$-configuration $N \cup \{t\}$ satisfies $a(N) \leq b$, so the only assumption violated in the definition of $(1,k)$-configuration inequality is the ``minimal cover" assumption.
$\hfill\diamond$
\end{example}

\begin{example}
Consider the knapsack constraint $(2,4,7,10,10,20)^\transp x \leq 30$. 
Consider $N=\{2,3,4,5\}$, $t = 6$, $k=2$. 
Then $a(N) = 31>30$, and $\sum_{i \in N}x_i \leq n':=3$ is valid for $K$, and for any $Q \subseteq N$ with $|Q| = 2$, $Q \cup \{t\}$ is a minimal cover. 
So $N \cup \{t\}$ is a generalized $(1,k)$-configuration and not a $(1,k)$-configuration.
The only assumption violated from the definition of $(1,k)$-configuration is $a(N) \leq b$. 
Furthermore, the corresponding generalized $(1,k)$-configuration inequality $\sum_{i=2}^5 x_i + 2x_6 \leq 3$ is facet-defining.
$\hfill\diamond$
\end{example}

We denote the separation problem associated with generalized $(1,k)$-configuration inequalities by \emph{G-CONFIG-SP}. 
The proof of the following \NP-hardness result about G-CONFIG-SP uses the same reduction technique as in the proof of Theorem~\ref{theo:CONFIG_reduction}. 

\begin{theorem}
\label{theo:G_CONFIG_reduction}
Problem G-CONFIG-SP is \NP-complete, even if $m=1$, or if $x^*$ is an extreme point solution to the LP relaxation~\eqref{eq:knapsack_relaxation}.
\end{theorem}

\begin{proof}
Using the same reduction as in the proof of Theorem~\ref{theo:CONFIG_reduction}, we want to show that:
given an input $(A,d,c,x^*)$ to CI-SP which is in $ (\Z^{m \times n}_+, \Z^m_+, \Z^n_+, [0,1]^n)$, we can find a corresponding input $(A', d', c', y^*)$ to G-CONFIG-SP such that there is a yes-certificate to CI-SP with input $(A,d,c,x^*)$ if and only if there is a yes-certificate to G-CONFIG-SP with input $(A', d', c', y^*)$.

Recall the construction in the proof of Theorem~\ref{theo:CONFIG_reduction}:
\begin{alignat*}{3}
A'_{ij} = A_{ij} \ \forall i \in [m], \forall j \in [n], &\qquad A'_{ik} = \sum_{j=1}^n A_{ij} \ \forall i \in [m], \forall k \in \{n+1,n+2\},\\
c'_j = c_j \ \forall j \in [n], &\qquad c'_{n+1} = c'_{n+2} = M,\\
d'_i = d_i + 2 \sum_{j=1}^n A_{ij} \ \forall i \in [m].
\end{alignat*}
As shown previously, when $M$ is chosen to be large enough, then $y^*:=(x^*,1,1)$ is an optimal solution to LP
\begin{equation}
\label{eq:lp_relaxation_config_g}
\max\{(c')^\transp y : A' y \leq d', y \in [0,1]^{n+2}\}
\end{equation}
if $x^*$ is an optimal solution to the linear program \eqref{eq:knapsack_relaxation}. 

Let $C \subseteq [n]$ be a cover with respect to the row constraint $a^\transp x \leq b$ of $A x \leq d$ such that the associated CI cuts off $x^*: x^*(C) > |C| - 1$. Let $C': = C \cup \{n+1,n+2\}$. By definition of the input $(A', d')$, $C'$ is a cover with respect to the row constraint $a'^\transp y = (a^\transp , a^\transp e, a^\transp e) y \leq b + 2 a^\transp e = b'$ within $A' y \leq d'$, and the associated CI cuts off $y^*: y^*(C') = 2 + x^*(C) > |C| + 1 = |C'| - 1$. 
As every cover inequality is also a generalized $(1,k)$-configuration inequality, we have obtained a generalized $(1,k)$-configuration inequality with respect to a row constraint within $A' y \leq d'$ that cuts off $y^*$.

To complete the proof, we must show that if $N \cup \{t\}$ is a generalized $(1,k)$-configuration with respect to some row constraint $a'^\transp y \leq b'$ within $A' y \leq d'$, and
$(n' - k + 1) y^*_t + \sum_{i \in N} y^*_i > n'$, for some $n'$ such that $\sum_{i \in N} x_i \leq n'$ is valid for $\{x \in \{0,1\}^{n+2}: a'^\transp y \leq b'\}$, then we can construct a cover $C$ with respect to the associated row constraint in $Ax \leq b$ such that the associated CI cuts off $x^*$.

By construction of $(A', d')$, we know that the row constraint $a'^\transp y \leq b'$ takes the form $(a^\transp , a^\transp e, a^\transp e) y \leq b + 2a^\transp e$, where $a^\transp x \leq b$ is a row constraint in $A x \leq b$. 
Observe that in the constraint $(a^\transp , a^\transp e, a^\transp e) y \leq b + 2a^\transp e$, any cover must contain both $n+1$ and $n+2$. By definition of generalized $(1,k)$-configuration, we know $\{n+1, n+2\} \subseteq N \cup \{t\}$.
Next, we continue our discussion by considering separately two cases.

\smallskip \noindent
\textbf{Case 1:} $t \notin \{n+1, n+2\}$. Then $\{n+1, n+2\} \subseteq N$. By definition of $(1,k)$-configuration, we know that, for any $Q \subseteq N$ with $|Q| = k$, $Q \cup \{t\}$ is a cover. 
This means $\{n+1, n+2\} \subseteq Q$ for any $Q \subseteq N$ with $|Q| = k$. Hence $k=|N|$, and we have 
\begin{equation}
\label{eq: g-config1}
(n' - |N| + 1) y^*_t + \sum_{i \in N} y^*_i > n'.
\end{equation}
Here, $n'$ is such that $\sum_{i \in N}x_i \leq n'$ is valid for $\{x \in \{0,1\}^{n+2}: a'^\transp y \leq b'\}$. 
Notice that $N \setminus \{n+2\}$ is not a cover to $a'^\transp y \leq b'$, so $n' \geq |N|-1$. 
If $n'=|N|-1$, then $N$ is a cover to $a'^\transp y \leq b'$, and \eqref{eq: g-config1} is simply $y^*(N) > |N|-1$. 
Note that $\{n+1, n+2\} \subseteq N$, and $a'^\transp = (a^\transp, a^\transp e, a^\transp e)$, $b'= b+2a^\transp e$, $y^*_{n+1}=y^*_{n+2} = 1$, hence $C:=N \setminus \{n+1,n+2\}$ is a cover to $a^\transp y \leq b$, and $y^*(C) = y^*(N)-2 > |N|-3 = |C| - 1$. 
If $n'=|N|$, then \eqref{eq: g-config1} gives $y^*_t + \sum_{i \in N} y^*_i > |N|$. 
Hence $C:=\{t\} \cup N \setminus \{n+1, n+2\}$ is a cover to $a^\transp y \leq b$, and $y^*(C) = y^*_t + \sum_{i \in N} y^*_i - y^*_{n+1} - y^*_{n+2} > |N|-2 = |C|-1$.

\smallskip \noindent
\textbf{Case 2:} $t \in \{n+1, n+2\}$. 
Since $a'_{n+1}=a'_{n+2} = a^\transp e$, without loss of generality, we assume $t=n+2$. 
Because for any $Q \subseteq N$ with $|Q| = k$, $Q \cup \{t\}$ is a cover, and any cover to $a'^\transp x \leq b'$ should contain both $n+1$ and $n+2$, we know that $k = |N|$ and $n+1 \in N$. Moreover, since $t = n+2 \notin N$, then $N$ is not a cover to $a'^\transp x \leq b'$. Hence $n' = |N|$. Therefore, the separating generalized $(1,k)$-configuration inequality gives:
\begin{equation}
\label{eq: g-config2}
y^*_{n+2} + \sum_{i \in N} y^*_i > |N|.
\end{equation}
Here $N \cup \{n+2\}$ is a cover to $a'^\transp \leq b'$. 
By definition of $a', b',$ and $y^*$, we know that $C:=N \setminus \{n+1\}$ is a cover to $a^\transp \leq b$, whose associated CI separates $x^*$. 

We have thereby shown that there is a yes-certificate to CI-SP with input $(A,d,c,x^*)$ if and only if there is a yes-certificate to G-CONFIG-SP with input $(A', d', c', y^*)$.
Together with Theorem~\ref{theo:dsp_hard}, our proof establishes that G-CONFIG-SP is \NP-complete, even if $m=1$, or if $x^*$ is an extreme point to the LP relaxation~\eqref{eq:knapsack_relaxation}. 
To conclude the proof, it suffices to realize that the input $y^* = (x^*,1, 1)$ for G-CONFIG-SP will be an extreme point of $\{y \in [0,1]^{n+2} : A'y \leq d'\}$ if $x^*$ is an extreme point of $\{x \in [0,1]^{n} : A x \leq d\}$.
\qed
\end{proof}

Since the definition of generalized $(1,k)$-configuration inequalities is simpler than the definition of $(1,k)$-configuration inequalities, from Theorem~\ref{theo:poly_solvable_config}, one can expect that G-CONFIG-SP is polynomially solvable when the fractional support of $x^*$ has relatively small size.

\begin{theorem}
\label{theo:poly_solvable_g_config}
Let $x^*$ be the input solution to G-CONFIG-SP. If $|\{i \in [n]: x^*_i \in (0,1)\}| \leq \alpha$, then a separating generalized $(1,k)$-configuration inequality can be obtained in $O(n^2 \cdot 2^\alpha \cdot \log n \cdot m)$ time, if one exists.
\end{theorem}

\begin{proof}
Let $N \cup \{t\}$ be a generalized $(1,k)$-configuration with respect to some row $a^\transp x \leq b$, whose corresponding inequality \eqref{eq: generalized_config} separates $x^*$:
\begin{equation}
\label{eq: fractional_support1}
(n'-k+1) x^*_t + x^*(N) > n'.
\end{equation}
We assume that $a_1 \leq a_2 \leq \ldots \leq a_n$. 
\begin{claim}
There exists $N^* \subseteq [n] \setminus \{t\}$ and $k' \in \N$, such that $N^* \cup \{t\}$ is a generalized $(1,k')$-configuration whose corresponding inequality \eqref{eq: generalized_config} separates $x^*$, and $\{i \in [n] \setminus \{t\}: x^*_i = 1\} \subseteq N^*$, $\{i \in N^*: x^*_i = 0\} = \emptyset$.
\end{claim}
\begin{cpf}
If there exists $i' \in [n] \setminus \{t\}$ with $x^*_{i'} = 1$ and $i' \notin N$, then for $N':=N \cup \{i'\}$, $N' \cup \{t\}$ is a generalized $(1,k+1)$-configuration. Since $x(N) \leq n'$ is valid for $\{x \in \{0,1\}^n: a^\transp x \leq b\}$, then $x(N') = x(N)+x_{i'} \leq n'+1$ is also valid for $\{x \in \{0,1\}^n: a^\transp x \leq b\}$. Therefore, we obtain a generalized $(1,k+1)$-configuration inequality $(n'-k+1) x_t + x(N') \leq n'+1$, which cuts off $x^*$ because of \eqref{eq: fractional_support1} and $x^*_{i'} = 1$. 

If there exists $i'' \in N$ such that $x^*_{i''} = 0$, then for $N'':=N \setminus \{i''\},$ we have $|N''| \geq k$. 
This is because, otherwise, $|N| = k$, and \eqref{eq: fractional_support1} gives $(n'-|N|+1) x^*_t + x^*(N'') > n'.$ 
Note that $x^*_t \leq 1$ and $x^*(N'') \leq |N''| = |N| -1$, thus we have $n' < (n'-|N|+1) x^*_t + x^*(N'') \leq n'-|N|+1 + |N|-1 = n'$, which gives a contradiction. 
Therefore, $N'' \cup \{t\}$ is a generalized $(1,k)$-configuration, whose corresponding inequality \eqref{eq: generalized_config} is $(n'-k+1) x_t + x(N'') \leq n'$, and which cuts off $x^*$ because of \eqref{eq: fractional_support1} and $x^*_{i''} = 0$. 

Applying the above arguments recursively, eventually we will end up with a set $N^* \subseteq [n] \setminus \{t\}$ and $k' \geq k$, such that $N^* \cup \{t\}$ is a generalized $(1,k')$-configuration whose corresponding inequality \eqref{eq: generalized_config} separates $x^*$.
\end{cpf}
From the above claim, without loss of generality, we can assume that $\{i \in [n] \setminus \{t\}: x^*_i = 1\} \subseteq N$ and $\{i \in N: x^*_i = 0\} = \emptyset$.
Therefore, G-CONFIG-SP can be solved by the following procedure:
\begin{enumerate}
\item Arbitrarily pick $t \in [n]$ and set $N_f \subseteq \{i \in [n] \setminus \{t\}: x^*_i \in (0,1)\}$. 
\item Let $N:=\{i \in [n] \setminus \{t\}: x^*_i = 1\} \cup N_f$. If $a(N) + a_t \leq b$, then stop.
\item Let $k$ be the smallest integer number such that $a_t + a(N_{[k]}) > b$, and let $n'$ be the smallest integer number such that $x(N) \leq n'$ is valid for $\{x \in \{0,1\}^n: a^\transp x \leq b\}$. Check if \eqref{eq: fractional_support1} holds.
\item If the answer to the previous check is yes for some $t$ and $N_f$, then the corresponding $N \cup \{t\}$ provides a yes-certificate to G-CONFIG-SP, since the associated generalized $(1,k)$-configuration inequality separates $x^*$; 
If the answer is no for all $t$ and $N_f$, then $x^*$ cannot be separated by any generalized $(1,k)$-configuration inequality from the knapsack constraint $a^\transp x \leq b$.
\end{enumerate}

Here step~2 runs in $O(n)$ time, step~3 runs in $O(n \log n)$ time.
Since $|\{i \in [n]: x^*_i \in (0,1)\}| \leq \alpha$, we have 
$$
|\left\{N_f: N_f \subseteq \{i \in [n] \setminus \{t\}: x^*_i \in (0,1)\} \right\}| \leq 2^\alpha.
$$
So the above procedure can be implemented in $O(n^2 \cdot 2^\alpha \cdot \log n)$ time, 
and we complete the proof by applying the above argument over all knapsack constraint $a^\transp x \leq b$. 
\qed
\end{proof}

Compared with the separating scheme in Theorem~\ref{theo:poly_solvable_config}, the above theorem suggests that it is in fact easier to separate the generalized $(1,k)$-configuration inequality, compared with the standard $(1,k)$-configuration inequality, even though the later one is a special case.
For generalized $(1,k)$-configuration inequalities, we also directly obtain the following corollary.

\begin{corollary}
\label{cor: g_config}
If $m$ is polylogarithmic in $n$ and $x^*$ is an extreme point solution to \eqref{eq:knapsack_relaxation}, then a separating generalized $(1,k)$-configuration inequality can be obtained in polynomial time, if one exists. 
\end{corollary}

\section{Weight Inequalities}

In this section, we show that WI-SP is \NP-hard, and we present special cases where it can be solved in polynomial time.
For a pack $P$ of a given knapsack constraint $a^\transp x \leq b$, we denote by $C(P): = \{i \in [n] \setminus P : a_i > r(P)\}$.
With this notation, the WI associated with $P$ takes the form
\begin{equation*}
\sum_{i \in P} a_i x_i + \sum_{j \in C(P)} (a_j-r(P)) x_j \leq a(P),
\end{equation*}
where we remind the reader that $r(P): = b - a(P)$.
We will need the following auxiliary result. 

\begin{lemma}
\label{lem:wi_sp}
Let $(a,b) \in \Z^{n+1}_+$ with $a([n])/b \notin \Z$, and let $x^*_1 = \ldots = x^*_n = b/a([n])$. 
Then there exists a pack $P$ of $a^\transp x \leq b$ whose associated WI separates $x^*$ if and only if there exists a pack $P'$ of $a^\transp x \leq b$ such that $r(P') > 0$, $P' \cup C(P') = [n]$, and $|C(P')| = \lfloor a([n])/b \rfloor$.
\end{lemma}

\begin{proof}
First, assume that there exists a pack $P'$ of $a^\transp x \leq b$ such that $r(P')>0$, $P' \cup C(P') = [n]$, and $|C(P')| = \lfloor a([n])/b \rfloor$. 
We have
\begin{align*}
    \sum_{i \in P'} a_i x_i^* + \sum_{j \in C(P')} (a_j - r(P')) x^*_j & = \sum_{i \in P' \cup C(P')} a_i x^*_i - r(P') \sum_{j \in C(P')} x^*_j \\
    & = \sum_{i \in [n]} a_i x^*_i - r(P') \sum_{j \in C(P')} x^*_j \\
    & = b - r(P') \cdot |C(P')| \cdot \frac{b}{a([n])} \\
    & > b - r(P') \\
    & = a(P'),
\end{align*}
where the inequality holds because $r(P')>0$, $a([n])/b \notin \Z$, and $|C(P')| = \lfloor a([n])/b \rfloor$.
Therefore, we know that the WI associated with pack $P'$ separates $x^*$.

Next, assume that there exists a pack $P$ of $a^\transp x \leq b$ whose associated WI separates $x^*$. 
Namely: 
$$
f(P): = \sum_{i \in P}a_i x^*_i + \sum_{j \in C(P)} (a_j - r(P)) x^*_j - a(P) > 0.
$$ 
Without loss of generality, we assume that $P$ has the largest value $f(P)$ among all packs of the knapsack constraint $a^\transp x \leq b$.
By re-arranging the terms and using some basic algebra, we have that $f(P) > 0$ implies $r(P) > 0$ and
\begin{equation*}
    \sum_{j \in C(P)} x^*_j < \frac{\sum_{i \in P \cup C(P)} a_i x^*_i - a(P)}{r(P)}.
\end{equation*}
Since $\sum_{i \in P \cup C(P)} a_i x^*_i \leq b$,
we obtain
\begin{equation*}
\sum_{j \in C(P)} x^*_j 
< \frac{b - a(P)}{r(P)} 
= 1.
\end{equation*}
Replacing $\sum_{j \in C(P)} x^*_j = |C(P)| \cdot b/a([n])$, we have shown
\begin{equation}
\label{eq:wi_must}
    |C(P)| < \frac{a([n])}{b}.
\end{equation}

Now let $i' \in P$.
Note that we have $r(P \setminus \{i'\}) = r(P) + a_{i'}$, so we obtain
\begin{align*}
f(P) - f(P \setminus \{i'\}) & = a_{i'} \left(\sum_{j \in C(P \setminus \{i'\})} x^*_j + x^*_{i'} - 1 \right) + \sum_{j \in C(P) \setminus C(P \setminus \{i'\})} \left(a_j - r(P)\right) x^*_j \\
& \in \left[a_{i'} \left( \sum_{j \in C(P \setminus \{i'\})} x^*_j + x^*_{i'} - 1 \right), \ a_{i'} \left( \sum_{j \in C(P)} x^*_j + x^*_{i'} - 1 \right)\right],
\end{align*}
where the last relation holds because $a_j > r(P)$ for every $j \in C(P)$ and $a_j - r(P) = a_j - r(P \setminus \{i'\}) + a_{i'} \leq a_{i'}$ for every $j \notin C(P \setminus \{i'\})$. 
Since $P \setminus \{i'\}$ is clearly also a pack, our maximality assumption on $f(P)$ implies that $f(P) \geq f(P \setminus \{i'\})$. 
Hence we have $a_{i'} \left( \sum_{j \in C(P)} x^*_j + x^*_{i'} - 1 \right) \geq 0$. 
This implies that $\sum_{j \in C(P)} x^*_j + x^*_{i'} \geq 1$ for any $i' \in P$. 
Since $x^*_1 = \ldots = x^*_n = b/a([n])$, we obtain
\begin{equation*}
\label{eq:wi_1}
|C(P)| \geq \frac{a([n])}{b} - 1.    
\end{equation*}
Combined with \eqref{eq:wi_must} and the assumption that $a([n])/b \notin \Z$, we have:
\begin{equation*}
\label{eq:wi_2}
    |C(P)| = \left\lfloor \frac{a([n])}{b} \right\rfloor.
\end{equation*}

To complete the proof, we only need to show that $P \cup C(P) = [n]$. 
If not, there exists $i' \in [n] \setminus (P \cup C(P))$, such that $P \cup \{i'\}$ remains a pack. 
Hence
\begin{align*}
f(P \cup \{i'\}) - f(P) & = a_{i'} \left(\sum_{j \in C(P)} x^*_j + x^*_{i'} - 1 \right) + \sum_{j \in C(P \cup \{i'\}) \setminus C(P)} \left(a_j - r(P \cup \{i'\})\right) x^*_j \\
& \in \left[a_{i'} \left( \sum_{j \in C(P)} x^*_j + x^*_{i'} - 1 \right), \ a_{i'} \left( \sum_{j \in C(P \cup \{i'\})} x^*_j + x^*_{i'} - 1 \right)\right].
\end{align*}
Since $|C(P)| = \lfloor a([n])/b \rfloor$ and $a([n])/b \notin \Z$, we know that $\sum_{j \in C(P)} x^*_j + x^*_{i'} > 1$, which implies that 
$f(P \cup \{i'\}) > f(P)$. 
This contradicts the maximality assumption on $f(P)$ and the fact that $P \cup \{i'\}$ is a pack. 
We have thereby shown $P \cup C(P) = [n]$.
\qed
\end{proof}

To prove that the separation problem WI-SP is \NP-hard, we establish a reduction from the \emph{Subset Sum Problem (SSP)} to WI-SP. \\

\noindent \textbf{Problem SSP}\\
\textbf{Input: }$\alpha \in \Z^n_+$ and $w \in \Z_+$.\\
\textbf{Question: }Is there a subset $S \subseteq [n]$ such that $\alpha(S) = w$?\\

The SSP is among Karp's 21 \NP-complete problems \cite{karp1972reducibility}.
It is simple to check that SSP is \NP-complete even if $w > \max(\alpha)$.
We are now ready to prove that WI-SP is \NP-hard.

\begin{theorem}
\label{theo:WI_complexity}
Problem WI-SP is \NP-complete, even if $m=1$, or if $x^*$ is an extreme point solution to the LP relaxation \eqref{eq:knapsack_relaxation}.
\end{theorem}

\begin{proof}
First, we prove the first part of the statement.
We show that WI-SP is \NP-hard even in case of a single knapsack constraint. 
Given an instance $(\alpha, w) \in \Z^{n+1}_+$ of SSP with $w > \max(\alpha)$, we construct a knapsack problem $\max\{c^\transp x : a^\transp x \leq b, x \in \{0,1\}^{2n+2}\}$ and give an optimal solution $x^*$ to the associated LP relaxation.
The data $a,b,c$ of the constructed knapsack problem is defined as follows:
\begin{align}
\label{eq:wi_reduction}
\begin{split}
& a_i := \alpha_i + 2, \qquad \forall i = 1, \ldots, n,\\
& a_{n+1} := w \cdot (n+1) + 2(n+1)^2 - 3n - \alpha([n]), \\
& a_{n+1+j} := 2, \qquad \forall j = 1, \ldots, n+1,\\
& b := w + 2n + 3,\\
& c := a,\\
& x^*_1 := \ldots := x^*_{2n+2} :
= \frac{w + 2n + 3}{w \cdot (n+1) + 2n^2 + 5n + 4}.
\end{split}
\end{align} 

It is simple to check that $a,b,c$ are all integral, that $(a,b,c,x^*)$ has polynomial encoding size with respect to that of $(\alpha, w)$, and that $a^\transp x^* = b$.
Furthermore, $x^*$ is an optimal solution to the knapsack problem described by \eqref{eq:wi_reduction}, since $c^\transp x^* = a^\transp x^* = b$. Hence $(a,b,c,x^*)$ is a feasible input to WI-SP where $m=1$. 
Note that $(w \cdot (n+1) + 2n^2 + 5n + 4)/(w + 2n + 3) = n+1 + 1/(w+2n+3) \notin \Z$.
Hence, we can apply Lemma~\ref{lem:wi_sp} and obtain that there exists a separating WI for $x^*$ if and only if there exists a pack $P$ such that:
\begin{equation}
    \label{eq:equiv_main_wi}
    \begin{split}
        & r(P) > 0, \qquad P \cup C(P) = [2n+2], \\
        & |C(P)| = \left\lfloor \frac{w \cdot (n+1) + 2n^2 + 5n + 4}{w + 2n + 3} \right\rfloor = n+1.
    \end{split}
\end{equation}

\begin{claim}
\label{claim: wi_key}
There exists a WI from constraint $a^\transp x \leq b$ that separates $x^*$ if and only if there exists a subset $S \subseteq [n]$ such that $\alpha(S) = w$.
\end{claim}

\begin{cpf}
It suffices to show that there exists pack $P$ such that \eqref{eq:equiv_main_wi} holds if and only if there exists a subset $S \subseteq [n]$ such that $\alpha(S) = w$.

First, we assume that $P$ is a pack such that \eqref{eq:equiv_main_wi} holds. 
The two equations in \eqref{eq:equiv_main_wi} imply $|P| = 2n+2 - |C(P)|=n+1$.
If $\{n+2, n+3, \ldots, 2n+2\} \cap C(P) = \emptyset$, then $P \cup C(P) = [2n+2]$ implies that $\{n+2, n+3, \ldots, 2n+2\} \subseteq P$, which means $P = \{n+2, n+3, \ldots, 2n+2\}$ since $|P| = n+1$. 
However, since $w > \max(\alpha)$, we know that $2 + \max(\alpha) + 2(n+1) \leq  w + 2n + 3 = b$, which implies that $P \cup \{i'\}$ is a pack for any $i' \in [n]$, and this contradicts the assumption $C(P) = [2n+2] \setminus P$ of \eqref{eq:equiv_main_wi}. 
Therefore, there must exist some $i' \in \{n+2, n+3, \ldots, 2n+2\} \cap C(P)$. 
Hence $r(P) = b-a(P) < a_{i'} = 2$. 
Moreover, because $r(P)>0$, we have $r(P) = 1$, which implies $a(P) = b-1 = w+2n+2$. 
Since $a_{n+1} = w \cdot (n+1) + 2(n+1)^2 - 3n - \alpha([n]) \geq w + 2(n+1)^2 - 3n + (w \cdot n - \alpha([n])) > w+2n+2$, we know $n+1 \notin P.$ 
Let $S: = P \cap [n]$.
We then obtain $a(S) = 2 |S| + \alpha(S)$ and $a(P \setminus S) = 2(|P| - |S|) = 2(n+1 - |S|)$. 
Therefore, $w+2n+2 = a(P) = a(S) + a(P \setminus S) = \alpha(S) + 2n+2$, which gives us $\alpha(S) = w.$

Next, we assume that $S$ is a subset of $[n]$ with $\alpha(S) = w$. 
Clearly, $n+1 \notin S$. 
Then, we define the set $\tilde{S}$ containing $n+1 - |S|$ arbitrary indices from $\{n+2, \ldots, 2n+2\}$. 
Then $P: = S \cup \tilde{S}$ is a pack such that \eqref{eq:equiv_main_wi} holds.
In fact, we have
\begin{align*}
r(P) & = b-a(P) \\
& = w+2n+3 - a(S) - a(\tilde S) \\
& = w+2n+3 - (2|S| + \alpha(S)) - 2(n+1-|S|)\\
& = 1.
\end{align*}
This further implies $C(P) = [2n+2] \setminus P$ and $|C(P)| = 2n+2 - |P| = n+1$, since $a_i > 1$ for all $i \in [2n+2]$. Hence \eqref{eq:equiv_main_wi} is satisfied by pack $P$.
\end{cpf}

Claim~\ref{claim: wi_key} completes the proof of the first part of the statement, since SSP itself is \NP-hard. 

\smallskip

Next, we prove the second part of the statement.
We show that WI-SP is \NP-hard, even if $x^*$ is an extreme point solution to the LP relaxation \eqref{eq:knapsack_relaxation}.
Given an instance $(\alpha, w) \in \Z^{n+1}_+$ of SSP with $w > \max(\alpha)$, 
we construct an instance of the multi-dimensional knapsack problem 
$\max\{c^\transp x : A x \leq d, x \in \{0,1\}^{2N}\}$ and give an optimal solution $x^*$ to the associated LP relaxation, where $N=2n+2$. 
Let $G$ be a node-node adjacency matrix of a cycle on $N$ nodes.
The constraints of the constructed multi-dimensional knapsack problem are then defined as follows:
\begin{equation}
\label{eq:wisp_extreme}
    \begin{split}
        & a^\transp y \leq b, \qquad Gz \leq e_N, \\
        & y_i + 2z_1 + 2z_2 + 2z_3 \leq 3 + \epsilon, \qquad \forall i \in [N].
    \end{split}
\end{equation}
Here $(a,b) \in \Z_+^{N+1}$ is defined as in \eqref{eq:wi_reduction}, $\epsilon := (w + 2n + 3)/(w \cdot (n+1) + 2n^2 + 5n + 4)$ and $e_N$ is the $N$-dimensional vector with all components equal to one.
Now we define the objective vector $c: = (a, e_N)$, and we let $x^* = (y^*, z^*): = (\epsilon e_N, e_N/2)$. 
Note that we can multiply all the rows of \eqref{eq:wisp_extreme} by $w \cdot (n+1) + 2n^2 + 5n + 4$ to get an instance of WI-SP with integral data. 
The instance defined here clearly has polynomial encoding size with respect to that of $(\alpha, w)$. 

We now verify that this is a valid input for WI-SP. 
Clearly $x^*$ is feasible.
Furthermore, by summing all inequalities in $Gz \leq e_N$, it follows that $x^*$ is an optimal solution to the LP relaxation.

Next, we show that $x^*$ is an extreme point of the polyhedron given by \eqref{eq:wisp_extreme}. 
Since $N=2n+2$ is even, then $G$ is a square matrix with rank $N-1$. 
We can further verify that the first $2N$ constraints in \eqref{eq:wisp_extreme} give a system of $2N$ linearly independent constraints in $2N$ variables, and the only vector that satisfies all of them at equality is $x^*$. 

\begin{claim}
\label{claim: wi_2}
There exists a WI from \eqref{eq:wisp_extreme} that separates $x^*$ if and only if there exists a WI from the constraint $a^\transp y \leq b$ that separates $y^*$.
\end{claim}

\begin{cpf}
First, we assume that $P$ is a pack with respect to some constraint $a'^\transp x \leq b'$ of \eqref{eq:wisp_extreme} such that its corresponding WI separates $x^*$. 
If such constraint $a'^\transp x \leq b'$ comes from the subsystem $Gz \leq e_N$, say $z_1 + z_2 \leq 1$, then the only WI is $z_1 + z_2 \leq 1$, which cannot be violated by $x^*$ since $x^*$ is a feasible point. If $a'^\transp x \leq b'$ is $y_i + 2z_1 + 2z_2 + 2z_3 \leq 3 + \epsilon$ for some $i \in [N]$, then all the nonempty packs that do not include variables with zero coefficient are $\{i\}, \{i, N+1\}, \{i, N+2\}, \{i, N+3\}, \{N+1\}, \{N+2\}, \{N+3\}$. 
The corresponding WIs are $y_i \leq 1$ and:
\begin{align*}
y_i + 2z_1 + (2-\epsilon) (z_2 + z_3) & \leq 3, & 2z_1 + (1-\epsilon) (z_2 + z_3) & \leq 2, \\
y_i + 2z_2 + (2-\epsilon) (z_1 + z_3) & \leq 3, & 2z_2 + (1-\epsilon) (z_1 + z_3) & \leq 2, \\
y_i + 2z_3 + (2-\epsilon) (z_1 + z_3) & \leq 3, & 2z_3 + (1-\epsilon) (z_1 + z_2) & \leq 2.
\end{align*}
It is simple to check that none of the above inequalities is violated by $x^* = (\epsilon e_N, e_N/2)$. 
Hence the constraint $a'^\transp x \leq b'$ is just $a^\transp y \leq b$. 
In other words, we have shown that if  \eqref{eq:wisp_extreme} admits a separating WI that separates $x^*$, then the constraint $a^\transp y \leq b$ admits a separating WI that separates $y^*$. 

On the other hand, any WI from the constraint $a^\transp y \leq b$ is also a WI from the entire linear system \eqref{eq:wisp_extreme}. 
We have thereby proven this claim.
\end{cpf}

Note that $y^*=\epsilon e_N$ in this proof coincides with the $x^*$ in Claim~\ref{claim: wi_key}.
From Claim~\ref{claim: wi_2} and Claim~\ref{claim: wi_key}, 
we have completed the proof for the second part of the statement of this theorem, since SSP is \NP-hard. 
\qed
\end{proof}

Even though the problem WI-SP is \NP-hard in general, in the next theorem we provide a special case where it can be solved in  polynomial time, and a separating WI can be obtained in polynomial time, if one exists.

\begin{theorem}
\label{theo:poly_solv}
Let $x^*$ be the input solution to WI-SP. If $\max\{|S|: x^*_i \in (0,1) \ \forall i \in S, x^*(S) < 1\}$ is bounded by a constant, then a separating WI can be obtained in polynomial time, if one exists. 
\end{theorem}

\begin{proof}
We assume without loss of generality that $A x \leq d$ has only a single constraint $a^\transp x \leq b$, since we can always solve WI-SP with input $(A,d,c,x^*)$ by solving the corresponding WI-SP problems for each single constraint individually. 
For any $P \subseteq [n]$, let $f(P): = \sum_{i \in P}a_i x^*_i + \sum_{j \in C(P)} (a_j - r(P)) x^*_j - a(P).$ Then $f(P) > 0$ implies 
\begin{equation*}
    \label{eq:final}
\sum_{j \in C(P)} x^*_j < \frac{\sum_{i \in P \cup C(P)} a_i x^*_i - a(P)}{r(P)} \leq \frac{b - a(P)}{r(P)} = 1.
\end{equation*}
Among all the packs with the largest $f(P)$ value, let $P'$ be one that is inclusion-wise maximal.
In other words, $f(P') \geq f(P)$ for any pack $P$, and $f(P)=f(P')$ implies that $P'$ is not contained in $P$.
Let $C: = C(P')$. 
Note that for any $i' \in P'$ we have
\begin{align*}
f(P') - f(P' \setminus \{i'\}) & = a_{i'} \left(\sum_{j \in C(P' \setminus \{i'\})} x^*_j + x^*_{i'} - 1 \right) + \sum_{j \in C \setminus C(P' \setminus \{i'\})} \left(a_j - r(P')\right) x^*_j \\
& \in \left[a_{i'} \left( \sum_{j \in C(P' \setminus \{i'\})} x^*_j + x^*_{i'} - 1 \right), a_{i'} \left( \sum_{j \in C} x^*_j + x^*_{i'} - 1 \right)\right].
\end{align*}
Here, the last inequality $f(P') - f(P' \setminus \{i'\}) \leq a_{i'} \left( \sum_{j \in C} x^*_j + x^*_{i'} - 1 \right)$ holds because $a_j \leq a_{i'} + r(P')$ for any $j \in C \setminus C(P' \setminus \{i'\})$. 
Since $P' \setminus \{i'\}$ is also a pack and $f(P') \geq f(P' \setminus \{i'\})$, we have
\begin{equation}
\label{eq:final_1}
\sum_{j \in C} x^*_j + x^*_{i'} \geq 1, \quad \forall i' \in P'.    
\end{equation}
On the other hand, for any $i' \in [n] \setminus (C \cup P')$:
\begin{align*}
f(P' \cup \{i'\}) - f(P') & = a_{i'} \left(\sum_{j \in C} x^*_j + x^*_{i'} - 1 \right) + \sum_{j \in C(P' \cup \{i'\}) \setminus C} \left(a_j - r(P' \cup \{i'\})\right) x^*_j \\
& \in \left[a_{i'} \left( \sum_{j \in C} x^*_j + x^*_{i'} - 1 \right), a_{i'} \left( \sum_{j \in C(P' \cup \{i'\})} x^*_j + x^*_{i'} - 1 \right)\right].
\end{align*}
Since $i' \in [n] \setminus (C \cup P')$, the set $P' \cup \{i'\}$ is still a pack, hence $f(P' \cup \{i'\}) - f(P') \leq 0$. 
Furthermore, since $P'$ is an inclusion-wise maximal pack with the largest $f(P')$ value, we have $f(P' \cup \{i'\}) - f(P') < 0$.
Therefore, 
\begin{equation}
\label{eq:final_2}
\sum_{j \in C} x^*_j + x^*_{i'} < 1, \quad \forall i' \in [n] \setminus (C \cup P'). 
\end{equation}
From \eqref{eq:final_1} and \eqref{eq:final_2}, we obtain 
\begin{equation}
\label{eq:final_3}
P' = \{i \in [n] \setminus C : x^*(C) + x^*_i \geq 1\}.
\end{equation}

We have thereby shown that there exists a WI from knapsack constraint $a^\transp x \leq b$ which separates $x^*$, if and only if there exists $C \subseteq [n]$, such that the corresponding $P'$, as defined in \eqref{eq:final_3}, is a pack satisfying $f(P')>0$. 
Therefore, WI-SP can be solved by checking whether the set $P' = \{i \in [n] \setminus C : x^*(C) + x^*_i \geq 1\}$ is a pack with $f(P')>0$, for any possible $C \subseteq [n]$ with $x^*(C) < 1$. 

Let $I_0 := \{i \in [n] : x^*_i = 0\}$ and $I_f := \{i \in [n] : x^*_i \in (0,1)\}$. 
From the assumptions of this theorem, we know that $\alpha: = \max\{|S| : x^*(S) < 1, S \subseteq I_f\}$ is a constant. 
For any $T \subseteq I_0$ and $S \subseteq I_f$ with $x^*(S) < 1$, it is easy to see that
$$
\{i \in [n] \setminus S : x^*(S) + x^*_i \geq 1\} = \{i \in [n] \setminus (S \cup T) : x^*(S \cup T) + x^*_i \geq 1\}.
$$
Hence, $\{i \in [n] \setminus C : x^*(C) + x^*_i \geq 1\}$ is a pack with positive $f$ value for some $C \subseteq [n]$ with $x^*(C) < 1$, if and only if $\{i \in [n] \setminus (C \setminus I_0) : x^*(C \setminus I_0) + x^*_i \geq 1\}$ is a pack with positive value, where $C \setminus I_0 \subseteq I_f$ and $x^*(C \setminus I_0) = x^*(C) < 1$. 
Therefore, WI-SP can be solved by the following procedure: 
\begin{enumerate}
    \item For any $S \subseteq I_f$ with $x^*(S) < 1$, construct the corresponding $P' = \{i \in [n] \setminus S : x^*(S) + x^*_i \geq 1\}$.
    \item Check if $P'$ is a pack with $f(P') > 0$.
    \item If the answer to the previous check is yes for some $S \subseteq I_f$ with $x^*(S) < 1$, then the corresponding $P'$ provides a yes-certificate to WI-SP, and its corresponding WI separates $x^*$; If the answer is no for all $S \subseteq I_f$ with $x^*(S) < 1$, then $x^*$ cannot be separated by any WI from the knapsack constraint $a^\transp x \leq b$.
\end{enumerate}
Since $\alpha = \max\{|S| : x^*(S) < 1, S \subseteq I_f\}$, we have 
$$
|\{S : x^*(S) < 1, S \subseteq I_f\}| \leq \sum_{k=0}^\alpha \binom{n}{k} = O(n^\alpha).
$$
So this above procedure can be implemented in polynomial time, 
and we complete the proof.
\qed
\end{proof}

In particular, Theorem~\ref{theo:poly_solv} implies that, if $x^*$ has a constant number of fractional components, then WI-SP can be solved in polynomial time. 
We directly obtain the following corollary.

\begin{corollary}
\label{cor: 2}
If $m$ is a constant and $x^*$ is an extreme point solution to \eqref{eq:knapsack_relaxation}, then a separating WI can be obtained in polynomial time, if one exists. 
\end{corollary}


\section*{Acknowledgments}

A.~Del Pia is partially funded by ONR grant N00014-19-1-2322. 
Any opinions, findings, and conclusions or recommendations expressed in this material are those of the authors and do not necessarily reflect the views of the Office of Naval Research.
The work of J.~Linderoth and H.~Zhu is supported by the Department of Energy, Office of Science, Office of Advanced Scientific Computing Research, Applied Mathematics program under Contract Number DE-AC02-06CH11347.

\bibliographystyle{plain}

\bibliography{cite}

\begin{thebibliography}{10}

\bibitem{atamturk2005cover}
Alper Atamt{\"u}rk.
\newblock Cover and pack inequalities for (mixed) integer programming.
\newblock {\em Annals of Operations Research}, 139(1):21--38, 2005.

\bibitem{balas1975facets}
Egon Balas.
\newblock Facets of the knapsack polytope.
\newblock {\em Mathematical Programming}, 8(1):146--164, 1975.

\bibitem{balas1978facets}
Egon Balas and Eitan Zemel.
\newblock Facets of the knapsack polytope from minimal covers.
\newblock {\em SIAM Journal on Applied Mathematics}, 34(1):119--148, 1978.

\bibitem{10.1007/978-3-030-73879-2_14}
Alberto Del~Pia, Jeff Linderoth, and Haoran Zhu.
\newblock Multi-cover inequalities for totally-ordered multiple knapsack sets.
\newblock In Mohit Singh and David~P. Williamson, editors, {\em Integer
  Programming and Combinatorial Optimization}, pages 193--207, Cham, 2021.
  Springer International Publishing.

\bibitem{10.1007/978-3-031-06901-7_13}
Alberto Del~Pia, Jeff Linderoth, and Haoran Zhu.
\newblock On the complexity of separation from the knapsack polytope.
\newblock In Karen Aardal and Laura Sanit{\`a}, editors, {\em Integer
  Programming and Combinatorial Optimization}, pages 168--180, Cham, 2022.
  Springer International Publishing.

\bibitem{ferreira1996solving}
Carlos~E Ferreira, Alexander Martin, and Robert Weismantel.
\newblock Solving multiple knapsack problems by cutting planes.
\newblock {\em SIAM Journal on Optimization}, 6(3):858--877, 1996.

\bibitem{gabrel2002scheme}
Virginie Gabrel and Michel Minoux.
\newblock A scheme for exact separation of extended cover inequalities and
  application to multidimensional knapsack problems.
\newblock {\em Operations Research Letters}, 30(4):252--264, 2002.

\bibitem{MR1656936}
Zonghao Gu, George~L. Nemhauser, and Martin W.~P. Savelsbergh.
\newblock Lifted cover inequalities for {$0-1$} integer programs: computation.
\newblock {\em INFORMS Journal on Computing}, 10(4):427--437, 1998.

\bibitem{MR1688131}
Zonghao Gu, George~L. Nemhauser, and Martin W.~P. Savelsbergh.
\newblock Lifted cover inequalities for {$0\text{-}1$} integer programs:
  complexity.
\newblock {\em INFORMS Journal on Computing}, 11(1):117--123, 1999.

\bibitem{MR1607353}
Christoph Helmberg and Robert Weismantel.
\newblock Cutting plane algorithms for semidefinite relaxations.
\newblock In {\em Topics in Semidefinite and Interior-Point Methods ({T}oronto,
  {ON}, 1996)}, volume~18 of {\em Fields Inst. Commun.}, pages 197--213. Amer.
  Math. Soc., Providence, RI, 1998.

\bibitem{hojny2019knapsack}
Christopher Hojny, Tristan Gally, Oliver Habeck, Hendrik L{\"u}then, Frederic
  Matter, Marc~E Pfetsch, and Andreas Schmitt.
\newblock Knapsack polytopes: a survey.
\newblock {\em Annals of Operations Research}, pages 1--49, 2019.

\bibitem{kaparis2010separation}
Konstantinos Kaparis and Adam~N Letchford.
\newblock Separation algorithms for 0-1 knapsack polytopes.
\newblock {\em Mathematical Programming}, 124(1-2):69--91, 2010.

\bibitem{karp1972reducibility}
Richard~M Karp.
\newblock Reducibility among combinatorial problems.
\newblock In {\em Complexity of Computer Computations}, pages 85--103.
  Springer, 1972.

\bibitem{kellerer2004multidimensional}
Hans Kellerer, Ulrich Pferschy, and David Pisinger.
\newblock Multidimensional knapsack problems.
\newblock In {\em Knapsack Problems}, pages 235--283. Springer, 2004.

\bibitem{klabjan1998complexity}
Diego Klabjan, George~L Nemhauser, and Craig Tovey.
\newblock The complexity of cover inequality separation.
\newblock {\em Operations Research Letters}, 23(1-2):35--40, 1998.

\bibitem{letchford2019lifted}
Adam~N Letchford and Georgia Souli.
\newblock On lifted cover inequalities: A new lifting procedure with unusual
  properties.
\newblock {\em Operations Research Letters}, 47(2):83--87, 2019.

\bibitem{padberg1975note}
Manfred~W Padberg.
\newblock A note on zero-one programming.
\newblock {\em Operations Research}, 23(4):833--837, 1975.

\bibitem{padberg19801}
Manfred~W Padberg.
\newblock (1, k)-configurations and facets for packing problems.
\newblock {\em Mathematical Programming}, 18(1):94--99, 1980.

\bibitem{puchinger2010multidimensional}
Jakob Puchinger, G{\"u}nther~R Raidl, and Ulrich Pferschy.
\newblock The multidimensional knapsack problem: Structure and algorithms.
\newblock {\em INFORMS Journal on Computing}, 22(2):250--265, 2010.

\bibitem{Schrijver86}
Alexander Schrijver.
\newblock {\em Theory of Linear and Integer Programming}.
\newblock John Wiley \& Sons, 1986.

\bibitem{weismantel19970}
Robert Weismantel.
\newblock On the 0/1 knapsack polytope.
\newblock {\em Mathematical Programming}, 77(3):49--68, 1997.

\bibitem{wolsey1975faces}
Laurence~A Wolsey.
\newblock Faces for a linear inequality in 0--1 variables.
\newblock {\em Mathematical Programming}, 8(1):165--178, 1975.

\end{thebibliography}

\end{document}